%% file: capacity-arxiv-v3.tex
\documentclass[12pt]{article}
\usepackage{amsmath,amsfonts,amssymb,amsthm}
\usepackage{hyperref}
\usepackage{graphicx,color}

\usepackage{chicago}
\renewcommand{\cite}{\shortciteN}

\normalbaselines
\newcommand{\eps}{\varepsilon}
\newcommand{\ttheta}{T}
\def\bea{\begin{eqnarray}}
\def\eea{\end{eqnarray}}
\newcommand{\bean}{\begin{eqnarray*}}
\newcommand{\eean}{\end{eqnarray*}}

\newcommand{\I}{{\bf 1}}
\newtheorem{proposition}{Proposition}[section]
\newtheorem{theorem}[proposition]{Theorem}

\newtheorem{lemma}[proposition]{Lemma}
\newtheorem{remark}[proposition]{Remark}
\theoremstyle{definition}

\numberwithin{equation}{section}
\newcommand{\nc}{\newcommand}
\nc{\0}{{\bf o}}
\nc{\dist}{{\rm dist}}
\nc{\bx}{{\bf x}}
\nc{\R}{{\mathbb R}}
\nc{\N}{{\mathbb N}}
\nc{\Z}{{\mathbb Z}}
\nc{\BP}{\mathbb{P}}
\nc{\BE}{\mathbb{E}}
\nc{\BQ}{F}
\nc{\bN}{{\mathbf N}}
\nc{\bM}{{\mathbf M}}
\nc{\bG}{{\mathbf G}}
\nc{\bU}{{\mathbf U}}
\nc{\BX}{{\mathbb X}}
\nc{\bg}{{\mathbf g}}

\nc{\cF}{\mathcal{F}}
\nc{\cB}{\mathcal{B}}

\DeclareMathOperator{\diam}{\mathrm{diam}}

\renewcommand{\phi}{\varphi}

\renewcommand{\emptyset}{\varnothing}

\begin{document}
\author{G\"unter Last\thanks{Karlsruhe Institute of Technology,
    Institut f\"ur Stochastik, Englerstrasse 2, D-76131 Karlsruhe,
    Germany. Email: \texttt{guenter.last@kit.edu}} 
\and Mathew D. Penrose\thanks{Department of Mathematical Sciences, University of Bath,
Bath BA2 7AY, United Kingdom,
Email: \texttt{m.d.penrose@bath.ac.uk}}
\and Sergei
Zuyev\thanks{Department of Mathematical Sciences, Chalmers University of 
Technology and University of Gothenburg, Gothenburg, Sweden. Email:
\texttt{sergei.zuyev@chalmers.se}}
}

\title{On the capacity functional of the infinite cluster
  of a  Boolean model}
\date{\today}

\maketitle

\begin{abstract}
  The original 2017 version of this paper \cite{LPZ17}
  contains a major
  gap in the proofs. In the subsequent publication \cite{LPZ24}
  we indicated how to fix this. For convenience of
  the reader, we here update the original paper to incorporate the
  suggested fix.
 
  \smallskip
  Consider a Boolean model in $\R^d$ with balls of random, bounded
  radii with distribution $F_0$, centered at the points of a Poisson
  process of intensity $t>0$.  The capacity functional of the infinite
  cluster $Z_\infty$ is given by
  $\theta_L(t) = \BP\{ Z_\infty\cap L\ne\emptyset\}$, defined for each
  compact $L\subset\R^d$.

  We prove for any fixed $L$ and $F_0$ that $\theta_L(t)$ is
  infinitely differentiable in $t$, except at the critical value
  $t_c$; we give a Margulis-Russo type formula for the derivatives.
  More generally, allowing the distribution $F_0$ to vary and viewing
  $\theta_L$ as a function of the measure $F:=tF_0$, we show that it
  is infinitely differentiable in all directions with respect to the
  measure $F$ in the supercritical region of the cone of positive
  measures on a bounded interval.

  We also prove that $\theta_L(\cdot)$ grows at least linearly at the
  critical value.  This implies that the critical exponent known as
  $\beta$ is at most 1 (if it exists) for this model.  Along the way,
  we extend a result of H.~Tanemura (1993), on regularity of the
  supercritical Boolean model in $d \geq 3$ with fixed-radius balls,
  to the case with bounded random radii.
\end{abstract}

\begin{flushleft}
\end{flushleft}


\section*{Foreword}

This is a corrected version of
our 2017 paper \cite{LPZ17}.
The main issue needing correction
is that Equations (4.4) and (4.5) of \cite{LPZ17}, the
so-called  {\em stabilization} and  {\em stopping radius}
properties of certain quantities $Z_{K}(\varphi,\varphi')$
and $R_{K,b}(\varphi,\varphi')$, 
can fail with the definitions of these quantities used  in \cite{LPZ17}.

Indeed, suppose for example that $b=1$, $K=\{\0\}$ (where $\0$ denotes
the origin in $\R^d$) and $\varphi$ is such that
$\dist(\0,Z(\varphi)) = \dist(\0,Z_\infty(\varphi)) \geq 99$), and
moreover $Z(\varphi')$ provides a `bridge' from $\0$ to
$Z_\infty(\varphi)$ (i.e., has a component that contains $\0$ and also
intersects $Z_\infty(\varphi)$).

Then using the definition in \cite{LPZ17}, we have
$Z_K(\varphi,\varphi')=\emptyset$ so $R_{K,b}(\phi,\varphi') =
1$. However, if we take $\psi= \varphi'|_{[B_9 ]}$ then both (4.4) and
(4.5) of \cite{LPZ17} fail for $n=1$.  A version of this example was first
suggested to us by Emil Schmitz.  In his Master's thesis,
\cite{Schmitz22} describes a way to fix this error in dimension
$d=2$. We briefly indicated how to fix
the issue in general dimensions in \cite{LPZ24} (this
work is independent of \cite{Schmitz22}). We present here a full
corrected version of our paper (the fix in \cite{LPZ24} itself contains
a gap, which we fill in here; see the remark after \eqref{e:A'}.)  We also took the opportunity to
correct some unrelated, relatively minor slips along the way.

\section{Introduction}
\label{secintro}

The Boolean model is a fundamental model of random sets in stochastic
geometry; see \cite{Hall88,MeesterRoy96,SKM,SW08}.  It is obtained by
taking the union $Z$ of a collection of (in general, random) compact
sets (known as grains) centered on the points of a homogeneous Poisson
process of intensity $t$ in $d$-space.  For a large class of grain
distributions, it is known that for $t$ above a critical value $t_c$
that is dependent on the grain distribution, the resulting random set,
denoted $Z(t)$, includes a unique infinite component, denoted
$Z_\infty(t)$.

The random set $Z_\infty=Z_\infty(t)$ is an important and fascinating object of
study. One way to investigate its distribution is through its {\em
  capacity functional}, defined as the set function $ L \mapsto
\theta_L(t):=\BP \{ Z_\infty(t) \cap L \neq \emptyset \}$, defined
for compact
$L \subset \R^d$.  If $L$ is a singleton, then
$\theta(t):=\theta_{\{0\}}(t)$ is called the {\em volume fraction} of
$Z_\infty(t)$, and in the case where the grains are all translates of a
fixed set $K_0$ (e.g.\ a unit ball), $\theta_{K_0}(t)$ is (loosely
speaking) the proportion of grains that lie in $Z_\infty$.  More
generally, the capacity functional of a random set and, in particular,
of $Z_\infty$, determines its distribution; see  \cite{SW08}.

In this article we investigate the capacity functional of $Z_\infty$
as a function of the intensity $t$.  
We consider the case where the grains are balls with random
radii with distribution $F_0$ for some probability measure $F_0$ on
$\R_+$ with bounded support.

  We show for any compact
$L\subset\R^d$ that $\theta_L(t)$ 
 is infinitely differentiable in $t$ for $t >t_c$ (it is
identically 0 for $t<t_c$), thereby adding to
earlier results on continuity of $\theta_L(t)$, $t > t_c$, and give an
explicit expression for the derivatives (Theorem \ref{diff}).  
More generally, allowing $F_0$ to vary and viewing
$\theta_L$ as a function of the measure $F:=tF_0$,
 we show (Theorem \ref{3diff}) that it
 is infinitely differentiable in all directions
with respect to the measure $F$ in the supercritical region
 of the cone of positive measures on a bounded interval.

We also prove (in Theorem \ref{texponent}) that
$\theta_L$ grows at least linearly in the right neighbourhood of the
threshold $t_c$. This is similar behaviour to that of the percolation
function
 in discrete percolation models; see
\cite[Ch.~5]{Gr99} and the references therein.
See \cite{Duminil15}
 for a  recent alternative proof of the discrete result,
under the assumption of non-percolation at the critical
point.  
It would be interesting to try to adapt this to the continuum.

In the course  of  proving the results 
mentioned above,
we show (in Theorem \ref{slabthm})
that if our Boolean model with random but bounded radii
is supercritical in $\R^d$ for $d \geq 3$, then
it is also supercritical in a sufficiently thick slab.  
Previously, only the case with fixed radii had been considered,
although the analogous result in the lattice is well known 
(\cite{GM}).
Also noteworthy is the fact that our proof of
Theorem \ref{texponent}  requires the continuum 
Reimer inequality 
 (\cite{GuRao99}).

We believe that our methods could also be used 
to give smoothness of the $n$-point
connectivity function as a function of $t$ for $t > t_c$.
We also expect similar methods to
be applicable for more general grains. See
Section \ref{secfinal} for further discussion.

\section{Preliminaries}
\label{secprelim}

Let $d \in \N$ with $d \geq 2$.  We shall be dealing with a stationary
(spherical) Boolean model in $\R^d$ which is described by means of a
(marked) point process. Consider the space $\BX:=\R^d\times\R_+$
(where $\R_+:=[0,\infty)$), equipped with the Borel $\sigma$-field
$\mathcal{B}(\BX)$ and the space $\bN$ of integer-valued locally
finite measures $\varphi$ on $\mathcal{B}(\BX)$.  
For $b \in (0,\infty)$, let $\bN^b$ be the space of all $\varphi \in \bN$
that are supported by $\R^d\times [0,b] $.  Let $\mathcal{N}$
denote the smallest $\sigma$-algebra of subsets of $\bN$ making the
mappings $\varphi\mapsto\varphi(D)$ measurable for all measurable
$D\subset \BX$.  It is often convenient to write $z\in\varphi$ instead
of $\varphi(\{z\})>0$.  

A {\em point process} on $\BX$ is then a measurable
mapping $\Phi$ from some probability space $(\Omega,\mathcal{F},\BP)$
into the measurable space $(\bN,\mathcal{N})$. It is convenient to fix
the mapping $\Phi$ and to consider for any locally finite measure
$\mu$ on $\mathcal{B}(\BX)$ a probability measure $\BP_\mu$ on
$(\Omega,\mathcal{F})$ such that the distribution
$\BP_\mu\{\Phi\in\cdot\,\}$ of $\Phi$ is that of a {\em Poisson process}
with \emph{intensity measure} $\mu$.
 This means that under $\BP_\mu$
the point process $\Phi$ has independent increments,
with $\Phi(D)$ Poisson distributed with mean $\mu(D)$, for each
 bounded $D \in \mathcal{B}(\BX)$.  See, e.g., \cite{Kallenberg} or 
\cite{LaPe16}.
Expectation under $\BP_\mu$ is denoted by $\BE_\mu$.

For $r >0, x \in\R^d$, we denote by $B_{r}(x)$ the closed Euclidean
ball of radius $r$ centered at $x$. Also we write $0$ for the origin of
$\R^d$ and $B_r$ for $B_r(0)$.

Any $\varphi\in\bN$ is of the form $\sum_i \delta_{z_i}=\sum_i \delta_{(x_i,r_i)}$, 
where the Dirac measure $\delta_z$ at $z\in\BX$ is defined
by $\delta_z(D)=\I\{z\in D\}$ for every $D\in\mathcal{B}(\BX)$.  
We then define
$Z(\varphi):=\cup_i B_{r_i}(x_i)$. 
 The balls $B_{r_i}(x_i)$ are referred
to as \emph{grains}.

Connected components of $Z(\varphi)$ are called \emph{clusters}. Given
$\phi\in\bN$, let $Z_\infty(\varphi)$ denote the union
of the unbounded connected components of $Z(\phi)$, i.e.\ of the
infinite clusters.

In this paper we deal with Poisson processes whose intensity
measure is of the form $\mu(d(x,r)):=dx\, F(dr)$,
where  $dx$ is the $d$-dimensional Lebesgue measure and
$F$ is a finite measure on $\R_+$ (not necessarily
 a probability measure).
When $\mu$ is of this form, we shall write $\BP_{F}$ for $\BP_{\mu}$
and $\BE_{F}$ for $\BE_{\mu}$. 
Also let $\Pi_F$ denote the distribution of $\Phi$
under $\BP_F$, i.e.\ 
the probability measure on $(\bN,\mathcal{N})$
given by $\Pi_F(\cdot) = \BP_F\{\Phi \in \cdot\}$.
 Set $|F| = F(\R_+)$, the total mass of $F$.
  Then $|F|$ is called the \emph{density} (or
\emph{intensity}) of the Poisson process under~$\BP_{F}$.
We shall assume that $F$ has  no atom at $\{0\}$;
in any case the singletons
do not contribute to percolation properties of $Z$ we study here.

Let $\bM$ (respectively $\bM_1, \bM_\pm$) denote the class of finite
non-zero Borel measures (respectively,
probability measures and finite signed measures) $F$ on 
$\R_+$ satisfying $F(\{0\})=0$.
Given  $b\in (0,\infty)$, we write $\bM^b$ (respectively $\bM_1^b$,
$\bM_{\pm}^b$) for the measures that are
supported by $[0,b]$, i.e.\ that satisfy $F((b,\infty))=0$.
Let
$\bM^\sharp:=\cup_{b \in (0,\infty)} \bM^b$, the 
measures with bounded support. Likewise, set
$\bM_1^\sharp:=\cup_{b \in (0,\infty)} \bM_1^b$ and
$\bM_\pm^\sharp:=\cup_{b \in (0,\infty)} \bM_\pm^b$.

Let $F \in \bM$.  Under $\BP_{F}$,
the set $Z:=Z(\Phi)$ is called a \emph{Boolean model}.
It can be constructed,
alternatively, by first generating an infinite independent sequence
$\{R_i\}$ from the probability distribution $F(\cdot)/|F|$, and then
placing balls of the corresponding radii at the points $\{X_i\}$
of a homogeneous Poisson point process with intensity $|F|$ in
$\R^d$. This equivalence stems from the independent marking property
of a Poisson process; for more details, see, e.g., \cite{Kallenberg} 
 or \cite[Ch.5]{LaPe16}.

The point process $\Phi$ is {\em stationary}
under $\BP_{F}$, which means that
for all $x\in\R^d$ we have
$\BP_{F}\{\ttheta_x\Phi\in\,\cdot\,\}=\BP_{F}\{\Phi\in\,\cdot\,\}$, 
where for any $\mu\in \bN$, the measure
$\ttheta_x\mu\in \bN$ is defined by
$\ttheta_x\mu(B\times C):=\mu((B+x)\times C)$, with $B+x:=\{y+x:\
y\in B\}$.
Hence $Z(\Phi)$ is stationary as well,
 that is $\BP_{F}\{Z+x\in\cdot\}$ does not depend
on $x,\ x\in\R^d$.
Since $Z_\infty(\Phi) + x = Z_\infty(T_{-x} \Phi)$ for all
$x \in \R^d$, $Z_\infty(\Phi)$ is also stationary.

The {\em volume fraction} of the
Boolean model is the probability that $Z$ covers a fixed point,
 for instance the origin 0, or in other words
 the proportion of  space covered by grains:
\begin{displaymath}
  \BP_{F}\{0\in Z\}=
1-\exp\Big[-
  \kappa_d\int r^d\,F(dr)\Big],
\end{displaymath}
where $\kappa_d:=\pi^{d/2}/\Gamma(d/2+1)$ stands for the volume of a $d$-dimensional unit ball.


Under $\BP_F$ the sets $Z(\Phi)$ 
and  $Z_\infty:=Z_\infty(\Phi)$  are
 almost surely closed. They are
  {\em random closed sets}, see \cite{Mol05} or \cite{SW08}.
Our primary object of
study here is $Z_\infty$.
 For each compact
$L\subset \R^d$, let
\begin{displaymath}
  \theta_L(F):=\BP_{F}\{L\cap Z_\infty\ne\emptyset\},
\end{displaymath}
so that $L\mapsto \theta_L(F)$ is the capacity functional of
$Z_\infty$ under $\BP_{F}$. As mentioned in Section \ref{secintro},
 the capacity
functional determines the distribution of $Z_\infty$.
In particular we set
\begin{displaymath}
  \theta(F):=\theta_{\{0\}}(F)=
\BP_{F}\{0\in Z_\infty\}=\BE_{F} |Z_\infty\cap[0,1]^d|,
\end{displaymath}
the 
 volume fraction
of $Z_\infty$ under $\BP_{F}$ (also called the
\emph{percolation function}).

By ergodicity (see \cite{MeesterRoy96}), 
if $\theta(F) >0$
then $\BP_{F}\{Z_\infty\ne\emptyset\} = 1$
and moreover the infinite cluster is $\BP_{F}$-a.s.\ unique
(i.e., $Z_\infty$ has only one connected component);
see  
 \cite[Theorem 3.6]{MeesterRoy96}.
In this case we say that {\em percolation occurs}.
Conversely,
if $\theta(F) =0$
then
$\BP_{F}\{Z_\infty\ne\emptyset\} =0$.
Setting $\bU$ to be the class
of $\phi \in \bN$ such that $Z(\phi)$ has
at most one unbounded component, we thus have
\begin{align} 
\BP_{F} \{ \Phi \in \bU \} =1, \mbox{ for any } 
F \in \bM.
\label{uic}
\end{align} 
Given $F \in \bM$ (not necessarily a probability measure),
consider the family of measures of the form $F^* = t F$ with $ t >0$.
By a coupling argument, $\theta(tF)$ is nondecreasing in $t$.  The
{\em critical value} (or {\em percolation threshold}) 
$t_c(F)$ is the supremum of those $t$ such
that $\theta(tF) = 0 $.
If $\int r^{d}\,F(dr) < \infty$ 
(for example, if $F \in \bM^\sharp)$),
then $0 < t_c(F) < \infty$;  
see \cite{Gouere08}.
 If $t_c(F) <1$ we say that
$F$ is {\em strictly supercritical}.

It is known that $\theta(tF)$
 is continuous in $t$ at least for $t \neq t_c(F)$, and 
  right-continuous for all $t$; see 
 \cite[Theorem 3.9]{MeesterRoy96}. 
For $d=2$, it is known
  \cite[Theorem 4.5]{MeesterRoy96} that
 $\theta(t_c(F) F)=0$ (and therefore $\theta(tF)$ is continuous for all $t$),
  and this is commonly believed to be true for $d \geq 3$ too.

\begin{remark}\label{rem:pf}\rm
For $r\ge 0$, the quantity $\BP_{F}\{B_r\subset
  Z_\infty(\Phi+\delta_{(0,r)})\} =\BP_{F}\{B_r\cap
  Z_\infty\ne\emptyset\}=\theta_{B_r}({F})$ can be interpreted as the
  conditional probability (under $\BP_{F}$)
 that $B_r$ belongs to the infinite cluster
  given that $(0,r)$ belongs to $\Phi$. Therefore $\int
  \theta_{B_r}({F})\, F(dr)/|F|$ is the conditional (Palm) probability that
  a typical grain (centered at the origin) is a part of the infinite
  cluster, i.e.\ the proportion of grains belonging to the unbounded
  connected component.
\end{remark}

Next we describe two important properties of Poisson
processes which we use in this paper. One is
the \emph{Mecke identity} (see e.g.\ \cite[Ch.4]{LaPe16}):
\begin{equation}
 \label{eq:mecke}
  \BE_\mu \int f(z,\Phi)\,\Phi(dz) 
  =\BE_\mu\int f(z,\Phi+\delta_{z})\,\mu(dz)
\end{equation}
for any measurable $f:\BX\times\bN\rightarrow \R_+$.
This identity characterises the Poisson process.

Another important result is the {\em perturbation formula}
for functionals of Poisson processes,
an analogue of the 
 Margulis-Russo formula for Bernoulli fields.
  For bounded measurable $f:\bN\rightarrow\R$, and $z \in \BX$,
define 
$D_zf(\varphi):= f(\varphi + \delta_z)
- f(\varphi)$, for all $\varphi \in \bN$. 
For $n\ge 2$ and $(z_1,\ldots,z_n)\in\BX^n$
we define a function
$D^{n}_{z_1,\ldots,z_n}f:\bN (\BX)\rightarrow\R$
inductively by
\begin{align}
\label{differn}
D^{n}_{z_1,\ldots,z_{n}}f:=D_{z_{1}}D^{n-1}_{z_2,\ldots,z_{n}}f.
\end{align}
The operator $D^n_{z_1,\ldots,z_n}$ is symmetric in $z_1,\ldots,z_n$;
indeed, by induction
\begin{align}
\label{differn2}
D^{n}_{z_1,\ldots,z_{n}}f(\varphi) = \sum_{I \subset \{1,\ldots,n\} }
(-1)^{n-|I|} f \left(\varphi + \sum_{i \in I} \delta_{z_i}  \right) 
\end{align}
where $|I|$ denotes the number of elements of $I$.

\begin{proposition}\label{russo}
  Let $\mu$ be a locally finite measure and 
  $\nu$ a finite signed measure on $\mathcal{B}(\BX)$.  Let
  $f:\bN\rightarrow\R$ be measurable and bounded. 
If $\mu + a \nu$ is a measure for some $a >0$,
  then
\begin{align}\label{deriv4}
  \left. \frac{d^+}{d s} 
\BE_{\mu+s\nu} f(\Phi)
\right|_{s=0}
  &=\int\BE_{\mu}
D_z f(\Phi)
\,\nu(dz).
\end{align}
If also $\mu - a \nu$ is a measure, then
$\BE_{\mu+s\nu} f(\Phi)$ is differentiable in $s$ at
$s=0$.
\end{proposition}
 The proof of this perturbation formula
 can be found in \cite[Theorem~2.1]{Zue:92b} (for
 the case $\nu=\mu$), 
for finite measures in \cite[Theorem~2.1]{MolZuy:00}, and for locally
finite measures and square-integrable functions in~\cite{Last12}.
It may also be found in \cite{LaPe16}.

\section{Main results}

\subsection{Smoothness of the capacity functional}
\label{secboolean}

Our first result concerns differentiating the capacity functional
$\theta_L$ with respect to the measure $F$. This can be useful to
compare the percolation properties of different radius
distributions. For example, in \cite{GouereMarchand11} and in
\cite{MRS} the
percolation threshold for $F$ a Dirac measure (i.e. balls of fixed
radius) has been compared with the percolation threshold for $F$ the
sum of two Dirac measures (i.e. for balls of random radius with just
two possible values), or with more general $F$.  
\cite{GouereMarchand11} show that in sufficiently high
dimensions the Dirac measure does not minimise the critical
volume fraction (as had been previously conjectured) but do not quantify the
phrase `sufficiently high' and do not rule out the possibility that
the Dirac measure minimises the critical volume fraction in low
dimensions.  With sufficient analytic tools, it might be possible to
compare different radius distributions (perhaps with the same
volume fraction) by calculus.
For example, we could compare two measures $F_1$ and $F_2$ by passing
continuously from one to the other.

Our result gives the directional derivative for $\theta_L(F)$ as we
vary $F$. If we wish to keep the total measure (i.e., the density)
constant then we need to add to $F$ a {\em signed} measure with total
measure zero. More generally we may consider adding an arbitrary
signed measure $G$ to $F$.  We use notation for the classes of measures from Section
\ref{secprelim} and $D^n$ from (\ref{differn}).
\begin{theorem}\label{3diff}
  Suppose that $F \in \bM^\sharp$ with $t_c(F) <1$,
  and $G \in \bM^\sharp_{\pm}$ is
  such that $F + a G$ is a measure for some $a >0$.  Let
  $L\subset\R^d$ be compact. Then
  \begin{align}
\nonumber 
    \left. 
    \frac{d^+}{d h}\theta_L(  F + hG )\right|_{h=0} 
    = \iint
    \BP_{F} \{L\cap Z_\infty(\Phi+\delta_{(x,r)})\ne\emptyset, 
    L\cap  Z_\infty(\Phi )=\emptyset\}
\\
\times
    G(dr)dx,\,
    \label{M2conj0}
  \end{align}
  and the right-hand side of (\ref{M2conj0}) is finite.  If also
  $F - a G$ is a measure, then $\theta_L(F + hG)$ is infinitely
  differentiable in a neighbourhood of $h=0$, 
then  setting
 $\tilde{f}_L(\varphi) = \I \{L \cap Z_\infty(\varphi) \neq \emptyset \}$
for $\varphi \in \bN$,
 for all $n \in \N$
  we have
  \begin{align}
    \left. 
    \frac{d^{n}}{dh^n} \theta_L(  F + hG )\right|_{h=0}
    = \idotsint \BE_{F}   D^n_{(x_1,r_1),\ldots,(x_n,r_n)} \tilde{f}_L (\Phi) 
\nonumber \\
    \times dx_1 G (dr_1) \cdots dx_n G(dr_n).
    \label{M2conj3}
  \end{align}
\end{theorem}

We shall prove Theorem~\ref{3diff} in Section~\ref{secpfdiff}.  The
identity (\ref{M2conj0}) tells us that the perturbation
 formula
(\ref{deriv4}) remains valid for
$f(\varphi) = \I \{ L \cap Z_\infty(\varphi) \neq \emptyset \}$ with
$\mu (dx dr) = dx \, F(dr)$ and $\nu (dx dr) = dx\, G(dr)$, though in this
case both $\mu$ and $\nu$ are {\em infinite} (but $\sigma$-finite).

Our next theorem is a corollary of Theorem~\ref{3diff}, and
significantly adds to the known results mentioned in Section
\ref{secprelim} concerning continuity of $\theta_L(tF)$ for fixed $F$,
in the case of deterministically bounded radii.
 Recall that the {\em
  Minkowski difference} $A\ominus B$ of two sets $A,B\subset\R^d$ is
defined by $\{x-y:\ x\in A,\ y\in B\}$. When $B=\{x \}$ for some 
$x \in \R^d$, we write simply $A - x$ for $A \ominus  \{x\} $.

\begin{theorem}\label{diff} Let 
$F \in \bM^\sharp_1$,
and let $L\subset\R^d$ be compact. Then $t \mapsto \theta_L(tF)$ is
infinitely differentiable on $(t_c(F),\infty)$ and setting
$\tilde{f}_L(\varphi) = \I \{ L \cap Z_\infty(\varphi) \neq \emptyset
\}$ for all $\varphi \in \bN$, we
have for $t >t_c(F)$ and $n \in \N$ that
\begin{align}
\frac{d^n \theta_L (tF) }{dt^n}  =
\int \cdots \int \BE_{tF} D^n_{(x_1,r_1)
, \ldots , (x_n,r_n) } \tilde{f}_L(\Phi) 
\nonumber \\
\times dx_1 F(dr_1)
\cdots
dx_n F(dr_n).
\label{M18}
\end{align}
In particular, for $t >t_c(F)$ we have
\begin{align}
  \frac{d}{dt}\theta_L(tF) 
=
\iint
\BP_{tF} \{L\cap Z_\infty(\Phi+\delta_{(x,r)})\ne\emptyset, 
L\cap  Z_\infty(\Phi)=\emptyset\}\,
\label{conj0}
\\
 \times dx\,F(dr)
\nonumber 
\\
~~~~~
~~~~~~~~~
  = \BE_{tF}\int
  |(Z_\infty(\Phi+\delta_{(0,r)})\ominus L)\setminus
  (Z_\infty(\Phi)\ominus L)|\,F(dr).
\label{conj}
\end{align}
\end{theorem}

\begin{proof}
  Let $L \subset \R^d$ be compact.  The infinite differentiability of
  $\theta_L(tF)$, and the formula (\ref{M18}) for
  $\frac{d^n}{dt^n} \theta_L(tF)$, follow from applying 
  Theorem \ref{3diff} to the measures $F^*$ and $G^*$ given by
  $F^* = tF$ and $G^* = F$; also (\ref{conj0}) follows as a special
  case of (\ref{M18}).  It remains to prove (\ref{conj}).  By
  stationarity, $T_x \Phi $ has the same distribution as $\Phi$ under
  $\BP_{tF}$, so the right side of (\ref{conj0}) equals
  \begin{align*}
    \BE_{tF}\iint\I\{(L-x)\cap
    Z_\infty(\Phi+\delta_{(0,r)})\ne\emptyset,(L-x)\cap
    Z_\infty(\Phi)=\emptyset\}\,dx\,F(dr)
\\
    = \BE_{tF}\iint\
\I\{x \in  (L\ominus Z_\infty(\Phi+\delta_{(0,r)}) ) \setminus
(L  \ominus Z_\infty(\Phi))\}\,dx\,F(dr),
  \end{align*}
  and then (\ref{conj}) follows from the fact that for any Borel sets
  $A,B,L$ we have
  $|(L \ominus A) \setminus (L \ominus B) | = |(A \ominus L) \setminus
  (B \ominus L) | $.
\end{proof}


\begin{remark}\label{r3.4}\rm 
  Making use of the Mecke identity \eqref{eq:mecke}, we can also
  rewrite (\ref{conj0}) as follows (see also \cite{Zue:92b}):
\begin{displaymath}
\frac{d}{dt}\theta_L(tF)=t^{-1} \BE_{tF} \int
\I\{L\cap Z_\infty(\Phi)\ne\emptyset, 
L\cap  Z_\infty(\Phi-\delta_{(x,r)})=\emptyset\}\,\Phi(d(x,r)).
\end{displaymath}
\end{remark}

\subsection{Bounds for the capacity functional}\label{sscaling}

Our next result provides a lower bound for the capacity functional of
the infinite cluster.
 This bound is linear
in the right neighbourhood of the critical value.

It is known for
lattice percolation models that the percolation function grows at
least linearly in the right neighbourhood of the threshold; see
\cite{ChayesCh86}, or
\cite{Gr99} and the references therein. 
Our result shows that this also holds for the spherical Boolean model.

\begin{theorem}\label{texponent} Let $b >0$.
Let $F \in \bM^b_1$ and let $L\subset\R^d$ be compact.
Set $t_c := t_c(F) $ and  $\alpha 
:=\BP_{t_c F}\{B_b\subset Z(\Phi)\}$.
  Then 
  \begin{align}
    \theta_L(tF)-\theta_L(t_cF) & \geq
    \frac {\alpha (t-t_c)\,
(1-\theta_L(tF)) }{ t},
\quad  t > t_c .
                                  \label{eq:rate1}
\end{align}
Furthermore,
\begin{align}
    \frac{\theta_L(tF)-\theta_L(t_c F)}{t-t_c} & \geq
    \frac{  \alpha (1-\theta_L(t_cF))}{t_c} + o(1)\quad\text{as}\
    t\downarrow t_c. \label{eq:rate2}
  \end{align}
\end{theorem}

Theorem \ref{texponent} is proved in Section \ref{texproof}.

The bounds~\eqref{eq:rate1}-\eqref{eq:rate2} also hold for the
integrated percolation functions $\int \theta_{B_r}(tF)\,F(dr)$; see
Remark~\ref{rem:pf}. 
For a given $F$,
an explicit numerical lower bound for the right-hand side of
 (\ref{eq:rate1}) can be established by using
the inequality
\begin{align*}
  1-\theta_L(tF)\geq \BP_{tF}\{L\cap Z=\emptyset\}=\exp\Big[-t
  \int |B_r\ominus L|\,F(dr)\Big]
\end{align*}
and applying a numerical estimation method for $t_c$
such as that in ~\cite{ZueSid:85b}, for example.  Also,
  it is not difficult to estimate
$\alpha$ (the probability that $B_b$ is fully covered) explicitly from
below.

\begin{remark}\label{remark4.4}\rm Although 
  the capacity functional $t \mapsto \theta_L(tF)$ is believed to be
 continuous at
  the critical value $t_c$, it is certainly not differentiable
  there.  Indeed, if it is continuous, then $\theta_L(t_cF)=0$ and the
  left-hand derivative of $t \mapsto \theta_L(tF)$ at $t_c$ 
vanishes. But Theorem
  \ref{texponent} implies that the right-hand derivative, if
it exists, is strictly
  positive.  
\end{remark}

\begin{remark}\label{remark4.3}\rm Given the common belief for
  discrete percolation (see \cite{Gr99}), one might conjecture that
  $\theta_L(tF)-\theta_L(t_cF)\sim (t-t_c)^\beta$ as $t\downarrow t_c$
  (at least in a logarithmic sense) for some {\em critical exponent}
  $\beta>0$. If this holds, Theorem \ref{texponent} implies $\beta\le 1$.
\end{remark}

\subsection{Percolation in a slab when $d \geq 3$}
\label{slabsec}

Given $K > 0$,
let $S(K)$ denote the slab $[0,K] \times \R^{d-1}$.
An important result of 
\cite{GM} says that for Bernoulli
 lattice percolation
 if the parameter $p$ is   
 supercritical in $\Z^d$, with $d \geq 3$,
then for sufficiently
large $K$ the parameter $p$ is
also supercritical for   
 the model restricted to a sufficiently
large slab in $\Z^d$.  

To prove our results 
 in the case $d \geq 3$,
 we need to adapt this
result to the Boolean model. In
the case where the balls have fixed radius,
this was done in \cite{Tanemura:93}, 
and we now describe an extension to balls
of random radius. This could potentially be of use elsewhere.

Given $F \in \bM$, let us denote by $\Phi_F$ a Poisson
process in $\BX$ with distribution $\Pi_F$,
i.e. with $\BP \{ \Phi_F \in \cdot\} =\BP_F\{\Phi \in \cdot\}$.
Given also any measurable function $f:\R_+ \to \R_+$,  let us denote
by $\Phi_{F,f(\rho)}$ the image of $\Phi_F$ under
the mapping $\sum_i \delta_{(x_i,r_i)} \mapsto \sum_i \delta_{(x_i,f(r_i)) }$. 
Thus $\Phi_{F,f(\rho)}$ has the same distribution
as $\Phi_{F \circ f^{-1}}$;  it will be convenient
for us to mention $\rho$ in
  the notation,
  representing the radius of a ball in the system.
For $\varphi \in \bN$ and
$A \in \cB(\BX)$ let $\varphi|_A$ denote the restriction of $\varphi$
to $A$, i.e.\  $\varphi|_A (\cdot) = \varphi(\cdot \cap A)$.  Finally,
for $B \subset \R^d$ write $[B]$ for $B \times \R_+$.

\begin{theorem}
\label{slabthm}
Suppose $d \geq 3$ and let $F \in \bM^\sharp$ with $t_c(F) <1$.  Then
there exists $K < \infty$ such that
$\BP_{F} \{ Z_\infty ( \Phi |_{[S(K)] }) \neq \emptyset \} = 1$, 
 and
\begin{align}
\inf_{x \in S(K)} \BP_F \{
x \in Z_\infty(\Phi|_{[S(K)]} ) \}  >0.
\label{slab1}
\end{align}
\end{theorem}

\begin{proof}
By assumption,
$F(\{0\})=0$. 
  By \cite[Theorem 3.7]{MeesterRoy96}, for any 
$b >0$ the value of $ t_c(F')$ depends
  continuously (in the weak topology) on $F' \in \bM_1^b$, so
  one can show that there exists $a > 0$ with $t_c(F|_{[a,\infty)}) < 1$,
 where $F|_{[a,\infty)}$ denotes the restriction of the
measure $F$
to  the interval $[a,\infty)$.
Since there exist coupled Poisson point  processes
 $\Phi,\Phi'$ having distribution $\Pi_F$ and 
$\Pi_{F|_{[a,\infty)}}$ respectively with
$\Phi' \leq \Phi$ almost surely,  it suffices
to prove (\ref{slab1}) using the measure
$F|_{[a,\infty)}$ rather than $F$. In other
words,
we may assume without loss of generality that there exists
  $a>0$ with $F([0,a))=0$, and then by scaling (see
  \cite{MeesterRoy96}) we can (and now do) assume $a = 1$.

  For $i=3,4,5$ choose $t_i$ with $t_c(F) < t_3 < t_4 < t_5 < 1$. Then
  $Z_\infty(\Phi_{t_3F } ) \neq \emptyset$ almost surely, so that by
  scaling, there exists $\delta >0$ with $1/\delta \in \N$ such that
  also
  $Z_\infty(\Phi_{t_4F,(1-\delta)\rho }) \neq \emptyset$ almost
  surely, and therefore also
  $Z_\infty(\Phi_{t_4F,\rho-\delta }) \neq \emptyset$
  since almost surely $\rho \geq 1$  so
  that $(1- \delta)\rho \leq \rho -\delta$.

  Set
  $\lfloor \rho \rfloor_\delta := \delta \lfloor \rho/\delta \rfloor$,
  i.e.\ the value of $\rho$ rounded down to the nearest integer
  multiple of $\delta$. Then
  $\lfloor \rho \rfloor_\delta \geq \rho -\delta$, so that
  $Z_\infty(\Phi_{t_4F, \lfloor \rho \rfloor_\delta}) \neq
  \emptyset$
  almost surely.  Note that since $1/\delta \in \N$ we have
  $\lfloor \rho \rfloor_\delta \geq 1$ almost surely.
  %
  By further scaling, we can (and do) choose $\eps >0$ such
  that
  $Z_\infty(\Phi_{t_5F,(1- 2 \eps) \lfloor \rho \rfloor_\delta})
  \neq \emptyset$ almost surely.

  Now let $\eta = \eps / (2d)$.  Divide $\R^d$ into half-open
  cubes denoted $Q_z, z\in \Z^d$, where $Q_z$ has side length $\eta$
  and is centered at $\eta z$.  For $x \in \R^d$ let $\langle x\rangle_\eta$ denote
  the point at the center of the cube $Q_z$ containing $x$. For
  $\varphi = \sum_i \delta_{(x_i,r_i)} \in \bN$, let
  $\langle \varphi \rangle_\eta := \sum_i \delta_{(\langle x_i\rangle_\eta,r_i) }$ (this
  counting measure can have multiplicities).  Since
  $|\langle x\rangle_\eta -x|\leq d \eta/2$ for all $x \in \R^d$, and since $\eta $
  is chosen so that $d \eta < \eps $, we have that
  $Z_\infty ( \langle \Phi_{t_5F,(1- \eps) \lfloor \rho \rfloor_\delta}
  \rangle_\eta) \neq \emptyset $ almost surely.

  For $t>0$, the occurrence of $Z_\infty^{} ( \langle \Phi_{tF,(1- \eps)
    \lfloor \rho \rfloor_\delta} \rangle_\eta) \neq
  \emptyset$ is equivalent to existence of an infinite cluster in the
  following Bernoulli site percolation model on $\Z^d
  \times \{1,2,\ldots,\kappa\}$ for some $\kappa \in
  \N$.  Let
  $r_1,\ldots,r_\kappa$
  denote the possible values for $\lfloor
  \rho \rfloor_\delta$ (where $\rho$ has distribution
  $F(\cdot)/|F|$),
  listed in increasing order.  For $1
  \leq i \leq \kappa$ set $\pi_i := \BP \{ (1-\eps) \lfloor \rho
  \rfloor_\delta = r_i \}$.  For $y,z \in \Z^d$ and $i,j \in
  \{1,\ldots,\kappa\}$, put $(y,i) \sim (z,j)$ if and only if $|\eta y
  - \eta z| \leq r_i + r_j$.  Let each site
  $(z,i)$
  be occupied with probability $p_{t,i}
  $, where we put $p_{t,i} = 1 - \exp(- t |F| \eta^d
  \pi_i)$, independent of the other sites.  Note that $(p_{1,i})_{i
    \leq \kappa}
  $ is supercritical, in that it strictly exceeds (in each entry) the
  vector $(p_{t_5,i})_{i \leq \kappa}$ which also percolates.

  By the result of \cite{GM} adapted to this site percolation model, 
there is a choice of $K$ such that
  $Z_\infty^{ } (\langle\Phi_{F,(1- \eps) \lfloor \rho \rfloor_\delta} |_{
    [S(K) ]}\rangle_\eta) \neq \emptyset $
  almost surely.  Therefore since $d \eta \leq \eps $ we have
  $ Z_\infty^{ } (\Phi_{F} |_{[S(K)
    ]}) \neq \emptyset $
  almost surely.  We may argue similarly to obtain (\ref{slab1}),
  following the proof of Lemma 10.8 in \cite{Pe03} with obvious
  modifications.

  Let us describe how to adapt some of the steps of \cite{GM} to the
  site percolation model above.  In Lemma 2 of \cite{GM}, we may
  replace $(1-p)^t$ by $(1- \max_i p_{\lambda_5,i})^t$.

  In Lemma 3 of \cite{GM}, instead of the box $B_m$ consider the box
  $B_{2m \lceil r_\kappa \rceil }$. Also the bound $(1-p)^{dk}$ would
  be replaced by $(1- \max_i p_{\lambda,i})^{-\kappa k B}$ where $B$
  denotes the number of sites of $\Z^d$ at distance at most
  $2 r_\kappa $ from the origin.

  In Lemma 4 of \cite{GM} we let $T(n)$ denote the set of sites
  $(z,i)$ lying in
  $B_n \cap \partial((-\infty,n] \times \Z^{d-1}) \times \{1,\ldots,
  \kappa\} $ with all coordinates of $z$ being nonnegative.
\end{proof}

\section{A preparatory result}
\label{secpreparatory}

In this section, we present further notation followed by a
key lemma (Lemma \ref{lint}) 
 which will be used repeatedly in the proof of
Theorems \ref{texponent} and \ref{3diff}.
For $A\subset\R^d$ and $\phi\in\bN$,
let $Z_A(\varphi)$ be
the union of all the clusters of $Z(\varphi)$ which have non-empty
intersection with $A$. In other words, 
 set
\begin{equation}
Z_A (\varphi) 
  := \bigcup_{i: x_i \leftrightarrow_\varphi A }  B_{r_i}(x_i), 
\quad \text{for}\ \varphi =  \sum_i \delta_{(x_i,r_i)} , 
\label{ZAdef}
\end{equation}
where $x \leftrightarrow_\varphi A$ means that
$x$ lies in a component of 
$Z(\varphi)$ which intersects
$A$.

 Recall that $\varphi|_A$ denotes the restriction of $\varphi\in\bN$
 to $A \in \cB(\BX)$ and $[B]=B \times \R_+$. 
Let $b \in (0,\infty)$. Given $\varphi, \varphi' \in \bN^b$,
after introducing
\begin{align*}
  R'_{K,b}(\varphi,\varphi') & := \min \{nb: n \in \N,
  Z_K(\varphi + \varphi') \subset B_{(n-1)b} \}; \\
 R''_{K,b} (\varphi,\varphi')
	& := \min \{ nb:n \in \N, 
 Z_K(\varphi
 + \varphi'|_{[B_{nb}]})
 \cap Z_\infty(\varphi) \neq \emptyset \}\\
	& \ = 
  \min \{ nb:n \in \N, K \cap Z_\infty(\varphi + \varphi'|_{[B_{nb}]})
 \neq \emptyset \},
\end{align*}
we define the
following `radius of stabilization':
\begin{equation}
  \label{rstabdef}
  R_{K,b} (\varphi,\varphi') :=
  \min(R'_{K,b} (\varphi,\varphi'),
  R''_{K,b} (\varphi,\varphi') ).
\end{equation}
Here we use the convention $\min(\emptyset):= + \infty$ if needed.
Also, below we shall sometimes write
 simply $R_K$ for $R_{K,1}$ and likewise
$R'_K, R''_K$ for $R'_{K,1}, R''_{K,1}$
respectively.


It is straightforward to verify 
that the stopping radius
property of $R_{K,b} $ (with respect to $\phi'$) holds, i.e.\  for each $n \in \N$,
if $R_{K,b}(\varphi,\varphi') = nb$ then
for any $\psi \in \bN_b$ we have
\begin{equation}
 R_{K,b}(\varphi, \varphi'|_{[B_{nb}]} + \psi|_{[\R^d \setminus B_{nb}]}) 
 = R_{K,b}(\varphi, \varphi'|_{[B_{nb}]}). 
 \label{eqstop}
\end{equation}
The notions of stabilization and of stopping radius have proved
fruitful in many other stochastic-geometrical contexts; see, for
example,~\cite{PenEJP}.

 For $F,G \in \bM$ we denote by 
$\Phi_F, \Phi'_G$ a pair of independent
 Poisson process with distribution
$\Pi_F$ and $\Pi_G$, respectively,
 i.e.\ with $\BP\{\Phi_F \in \cdot\,\} = \BP_F\{\Phi \in \cdot\,\}$
and
$\BP\{\Phi'_G \in \cdot\,\} = \BP_G\{\Phi\in\ \cdot\,\}$.
Recall from Section \ref{slabsec} the definition of 
the slab $S(K), K >0$.


\begin{lemma}
 \label{condcross}
 Suppose $d \geq 3$, $ b >0$ and $F \in \bM^b$,
 $G \in \bM^b_{\pm}$ and $\eps \in (0,1)$ are such that
 $ F + \eps G \in \bM$, and $ 1-\eps > t_c (F) $.  Then there
 exists $K \in (2b,\infty)$ and $\gamma \in (0,1/2)$ such that for all
 Borel $A \subset \R^d$, and
 $h \in (0, \eps^2/2]$,
\begin{multline}
\max\bigl(  \BP\{
  Z_A(\Phi_{F + h G}|_{[S(K)]}  ) = \emptyset\} 
, \\
\BP \{
Z_\infty (\Phi_{(1-\eps)F }|_{[S(K) ]} )
\cap Z_A( \Phi'_{\eps F + h G}|_{[S(K) ]} ) \neq \emptyset  \} \bigr)
 \geq \gamma.
\label{0701b}
\end{multline}
\end{lemma}
\begin{proof}
Choose $K$ as in Theorem \ref{slabthm}.
 Assume without loss of generality that $K \geq 2b$.

Suppose
 $\BP \{
  Z_A (\Phi_{F + h G}|_{[S(K)]}  ) \neq \emptyset \} \geq
 1/2$ (otherwise (\ref{0701b}) is immediate). 
Since $\BP\{ Y > 0 \} \geq p \BP \{ X > 0 \} $ for
any $p \in (0,1) $ and Poisson variables $X,Y$ with
$\BE Y = p \BE X$ (by Bernoulli's inequality), 
 also
 $\BP \{
  Z_A (\Phi_{(\eps/2)(F + h G)}|_{[S(K)]}  ) \neq \emptyset \} \geq
 \eps/4$. 
For $0 \leq h \leq \eps^2 /2$, we have
\begin{align*}
\eps F + hG - (\eps/2) (F+hG) = (\eps/2)
[ F + (2 h/ \eps) (1- \eps/2) G] \in \bM. 
\end{align*}
Hence, also
$\BP \{ Z_A (\Phi'_{\eps F + h G}|_{[S(K) ]} ) \neq \emptyset \} \geq
\eps/4$.
Given $ Z_A (\Phi'_{\eps F + h G}|_{[S(K) ]} ) \neq \emptyset$, the
set $Z_A( \Phi'_{\eps F + h G}|_{[S(K) ]} ) $ has non-empty
intersection with $S(K)$, and therefore by Theorem \ref{slabthm} and
our choice of $K$, the conditional probability that this set
intersects with $Z_\infty(\Phi_{(1-\eps) F}|_{[S(K) ]} )$ is bounded
below by a strictly positive constant $\gamma_1$.  Hence we have
(\ref{0701b}) with $\gamma = \gamma_1 \eps/4$.
\end{proof}

Recall from Section \ref{secprelim} that $ \bM^\sharp$ denotes
the class of measures on $(0,\infty)$ with bounded support.
We now give the main result of this section.
\begin{lemma}\label{lint} 
Let $b \in (0,\infty)$, and
 suppose that $F \in \bM^b$, $G \in
\bM^b_{\pm}$, and
$\eps \in (0,1)$ are such that $t_c(F) < 1- \eps$ and $F+ \eps G$ is a
measure. Then for any compact $L \subset \R^d$, we have
\begin{equation}
  \limsup_{n \to \infty} \sup_{
  0 \leq h \leq \eps^2/2 }
  n^{-1}  \log \BP
  \{ R_{L,b}(\Phi_{(1 -\eps) F},\Phi'_{\eps F + h  G } )
  > n \} < 0. 
\label{MPL1}
\end{equation}
Also, with $0$ denoting the zero
measure,
\begin{equation}\label{MPL2}
 \limsup_{n \to \infty}
  n^{-1}  \log
  \BP  \{ nb <  R_{L,b}(0,\Phi_{F} ) < \infty  \} < 0. 
\end{equation}
\end{lemma}
The $\eps^2$ in the range of $h$ in
 (\ref{MPL1}) arises because we need $h \leq \eps^2$ to guarantee
that $\eps F + hG$
is a measure. The fact that it is $\eps^2/2$ rather than $\eps^2$
 in (\ref{MPL1}) is an artefact of the proof. 

\begin{proof}[Proof of Lemma~\ref{lint}] 
Let $F,G,\eps$ be as in the statement of Lemma \ref{lint}.
First suppose  $d=2$.
Using Corollary 4.1 of \cite{MeesterRoy96} as a starting-point,
we can adapt  the proof of Lemma~10.5 of \cite{Pe03}  to random 
radius balls, thereby  showing that the probability
that
$Z(\Phi_{(1- \eps)F})$ fails to cross the rectangle $[0,3a] \times [0,a]$
decays  exponentially in $a$.

Given $a >0$, specify a sequence $D_1(a),D_2(a),\ldots$ of rectangles of
aspect ratio 3, alternating between horizontal and vertical rectangles,
 with $D_1(a) = [0,3a] \times [0,a]$,
 and with $D_{n}(a)$ crossing $D_{n+1}(a)$ the
short way for each $n$.  By the union bound and the exponential decay
just mentioned, the probability that 
for some $n$ there is no long-way crossing of $D_n(a)$ in 
$Z(\Phi_{(1- \eps) F})$ decays exponentially in $a$.  Hence  
the probability that 
$Z_\infty(\Phi_{(1-\eps)F})$ fails to include a
 long-way  crossing of $[0,3a] \times [0,a]$
is exponentially decaying in $a$. Likewise for the vertical rectangle
$[0,a] \times [0,3a]$.

Let $E_a$ denote the event that $Z_\infty(\Phi_{(1-\eps)F})$ includes
long-way crossings of each of the rectangles
$[-3a/2,-a/2] \times [-3a/2 , 3a/2]$,
$[a/2, 3a/2] \times [-3a/2,3a/2]$, $[-3a/2,3a/2] \times [-3a/2,-a/2]$
and $[ -3a/2,3a/2] \times [a/2,3a/2]$ (whose union is the annulus
$[-3a/2,3a/2]^2 \setminus (-a/2,a/2)^2$).  By the preceding
discussion, $1- \BP[E_a]$ decays exponentially in $a$.

Suppose $E_a$ occurs for some $a$ with
$L \subset [-a/2,a/2]^2$. Let $h \geq 0$. If
$R'_{L,b}(\Phi_{(1-\eps)F},\Phi'_{\eps F+hG})  > 3a +b,$
then we must have $Z'_{L}(\Phi_{(1-\eps)F}+ \Phi'_{\eps F+hG}|_{B_{3a+b}})  
\cap Z_\infty(\Phi_{(1-\eps)F}) \neq \emptyset$, and
therefore
$R''_{L,b}(\Phi_{(1-\eps)F},\Phi'_{\eps F+hG})  \leq 3a + 2b$.
Hence 
$R_{L,b}(\Phi_{(1-\eps)F},\Phi'_{\eps F+hG})  \leq 3a +2 b$
whenever $E_a$ occurs, and the case $d=2$ of (\ref{MPL1})
 follows.


 Now consider $d \geq 3$.
	Suppose $0\leq h \leq \eps^2/2$. 
Choose $b \in (0,\infty)$ with $F \in \bM^b$ and
$G \in \bM_{\pm}^b$.
Set $\Phi''_{F + hG} := \Phi_{(1-\eps)F} + \Phi'_{\eps F + hG}$.
Define the event
	$A_{u} : = \{ R_{L,b} (\Phi_{(1-\eps)F},\Phi'_{\eps F + h G}) 
	> u \}$. 
It suffices to prove that $\BP( A_{u} )$
decays exponentially in $u$, uniformly over
$h \in [0, \eps^2/2]$.
We now fix such $h$.
Let $K$ be as in Lemma~\ref{condcross}, and choose 
$n_0 \in \N$ with  $L \ominus B_{2b}  \subset [-n_0K,n_0K]^d$.

For $n \in \Z$ let $S_{n}$ denote the slab
$((n-1)K,nK] \times \R^{d-1}$ and let $H_n$ denote
the half-space $\cup_{-\infty <m \leq n} S_m$. 
Given $u $ with $L \subset B_{u-b}$ and
 $n \geq  n_0$,
for $\varphi, \varphi' \in \bN^b$
set
\begin{displaymath}
W_{u,n} (\varphi, \varphi') = 
	Z_L
	((\varphi +\varphi')|_{
[H_n \cap B_u]}  ) 
\end{displaymath}
and define the indicator functions
\begin{align*}
  f_{u,n} (\varphi,\varphi') & := \I \{ 
			       Z((\varphi +\varphi')|_{[S_{n+1} ]})  
                               \cap 
                               W_{u,n} (\varphi,\varphi') 
                               \neq \emptyset \};   
  \\
  g_{u,n}(\varphi,\varphi') & := \I \{ 
                              Z_\infty ( \varphi|_{[S_{n+1} ]} ) 
                              \cap 
                              W_{u,n}(\varphi,\varphi') 
                              = \emptyset \};   
  \\
	h_{u,n}(\varphi,\varphi') & := \I \{ Z_{L} 
	((\varphi+ \varphi')|_{[B_u]} ) \setminus
			      H_n \neq \emptyset \} 
			      \I \{R''_{L,b}(\varphi,\varphi') > u\}.
\end{align*}
Then we claim that $f_{u,n+1} (\varphi,\varphi') \leq 
f_{u,n}(\varphi, \varphi')$.
Indeed,
\begin{align}
 &\mbox{if }  f_{u,n}(\varphi, \varphi') =0 \label{M82} \\ 
 &\mbox{ then } W_{u,n+1}(\varphi,\varphi') = W_{u,n}(\varphi,\varphi') \subset
(-\infty, nK +b] \times \R^{d-1}, \notag
\end{align}
so that
it is not possible for
$ Z((\varphi +\varphi')|_{[S_{n+2} ]}  )$
to intersect with $W_{u, n+1}(\varphi,\varphi') $.
 
Also for integers $m,n$ with $n_0 +1 \leq m < n < (u/K)-2$ we have
$h_{u,n+1}(\varphi, \varphi') \leq
f_{u,m}(\varphi,\varphi')
g_{u,m}(\varphi,\varphi')$
by (\ref{M82}).
Therefore 
 \begin{align}
   h_{u,n}(\varphi,\varphi') 
   \leq \prod_{m=n_0}^{n-2} f_{u,m}(\varphi,\varphi') g_{u,m}(\varphi,\varphi').
\label{M83}
 \end{align}

Denote by $\cF_n $ the $\sigma$-field generated by 
$(\Phi_{(1-\eps)F}|_{[H_n ]}, 
\Phi'_{\eps F + h G}|_{[H_n ]})$. If the conditional expectation of
 $f_{u,n}(\Phi_{(1-\eps)F},\Phi'_{\eps F + h G} ) $
 with respect to $\cF_n$ is at least 1/2, then by 
Lemma 
\ref{condcross}
with $A$ taken to be
$W_{u,n}(\Phi_{(1-\eps)F},\Phi'_{\eps F + h G})$
the conditional expectation of
 $1-g_{u,n} (\Phi_{(1-\eps)F}, 
\Phi'_{\eps F + h G})$ 
is at least $\gamma$.
(Lemma~\ref{condcross} is applicable when $A$ is a random closed set.)
Hence, setting
\begin{displaymath}
  V_{u,n} := 
	f_{u,n}(\Phi_{(1-\eps)F},\Phi'_{\eps F + h G})
	g_{u,n}(\Phi_{(1-\eps)F},\Phi'_{\eps F + h G}),
\end{displaymath}
we have
\begin{displaymath}
  \BE [ V_{u,n} \, |\,\cF_n] \leq 
  \max(1-\gamma,1/2) = 1- \gamma.
\end{displaymath}
Also, for each $n$, $V_{u,n+1}$ is $\cF_n$-measurable and
by (\ref{M83}) we have
 \begin{align}
	 \BE[h_{u,n+2}(\Phi_{(1-\eps)F},\Phi'_{\eps F + h G})]
\leq \BE \left[ \prod_{i=n_0}^{n} V_{u,i} \right] \nonumber \\
= \BE  \left[ \BE[ V_{u,n} | \cF_{n}] \times \prod_{i= n_0}^{n-1} V_{u,i} \right]
 \leq (1-\gamma) \BE  \prod_{i=n_0}^{n-1} V_{u,i} 
\nonumber \\
\leq \cdots \leq (1-\gamma)^{n-n_0}  \leq 
\exp(- \gamma(n-n_0)).
	 \label{0429a}
 \end{align}

 For each $i \in \{1,\ldots,d\}$ let $H_n^{i,+}:=\{x \in \R^d:\pi_i(x) 
 \leq Kn\}$
 and $H_n^{i,-}:=\{x \in \R^d:\pi_i(x) \geq -K n\}$, where
 $\pi_i$ denotes projection onto the $i$th coordinate, set
 $$
	h_{u,n}^{i,+}(\varphi,\varphi')  := \I \{ Z_{L} 
	((\varphi+ \varphi')|_{[B_u]} ) \setminus
			      H_n^{i,+} \neq \emptyset \}
			      \I \{R_{L,b}(\varphi,\varphi') > u\},
			      $$
	define $h_{u,n}^{i,-}(\varphi,\varphi') $ similarly.

	Now set $\varphi:= \Phi_{(1-\eps)F}$ and
	$\varphi':=\Phi'_{\eps F+ hG}$.
 Suppose $A_u$ occurs; then  
 $R'_{L,b}(\varphi, \varphi')
 > u$ and therefore 
 $Z_L(\varphi+ \varphi') \setminus
 B_{u-b} \neq \emptyset$. This implies that for
 some $i \leq d$ either
 $Z_L((\varphi + \varphi')|_{[B_u]} ) \setminus H_{\lfloor u/dK \rfloor}^{i,+} 
 \neq \emptyset$ or 
 $Z_L((\varphi+ \varphi')|_{[B_u]} ) \setminus
 H_{\lfloor u/dK \rfloor}^{i,-} \neq \emptyset$;
 also
 $R^{''}_{L,b}(\varphi, \varphi') >u$
 and therefore
  $h_{u,n}^{i,+} (\varphi,\varphi')=1$ or
  $h_{u,n}^{i,-} (\varphi,\varphi')=1$  for some $i \in \{1,\ldots,d\}$.
 Therefore by applying the argument leading to (\ref{0429a})
 in each of the $2d$ positive or negative coordinate directions,
 and then the union bound,
for $u$ a multiple of $b$ we have
\begin{displaymath}
  \BP ( A_{u} )
  \leq 2d \exp(- \gamma(\lfloor (u-b)/(dK) \rfloor -n_0 -2)),
\end{displaymath}
which gives us the desired result (\ref{MPL1}) for $d \geq 3$.

To deduce \eqref{MPL2}, assume without loss of generality that
$\Phi_F$ is given as
the superposition of $\Phi_{(1-\eps)F}$ and $\Phi'_{\eps F}$.
If
$ nb<  R_{L,b}(0,\Phi_{F} ) < \infty  $,
then $R_{L,b}(\Phi_{(1-\eps)F},\Phi'_{\eps F} ) > nb$,
so (\ref{MPL2}) will follow from taking $h=0$ in (\ref{MPL1}).
\end{proof}


\section{Proof of Theorem \ref{3diff}}
\label{secpfdiff}

Suppose that  $b \in \R_+$ and
$F \in \bM^b$ and $G \in \bM_\pm^b$,
with $t_c (F) <1$ and $F+ aG$ a measure for some $a >0$.

To ease notation, we shall assume 
additionally that $b=1$; 
the result for a general $b$ can be obtained
by using the scaling property of the Boolean model, see,
e.g.,~\cite{MeesterRoy96}.

Choose $\eps \in (0,1)$ with $1-\eps > t_c( F)$ and with $F +\eps G$ a
measure. Keep $F,G $ and $\eps$ fixed for the rest of this section.

Let $G_+ $ and $ G_-$ be the positive and negative parts in the
Hahn-Jordan decomposition of $G$ (so that $G_+$ and $G_-$ are mutually
singular measures and $G = G_+ - G_-$).  Let $h \in [0,\eps^2]$. Then
$\eps F - h G_-$ is a measure.
Recall from Section \ref{secprelim} that
$\Pi_F$ denotes the distribution of a Poisson
process on $\R^d \times \R_+$
with intensity measure $\mu(d(x,r)) = dx \,  F(dr)$.  
Let $ \Phi_{(1-\eps)F} $,
$\Psi_{\eps F - h G_-}$, $\Psi_{h G_-}$ and $\Psi_{h G_+}$ be
independent Poisson processes in $\R^{d} \times \R_+$ with respective
distributions
 $\Pi_{(1-\eps) F }$,
$\Pi_{\eps F - h G_-}$, 
$\Pi_{h G_-} $ and
$\Pi_{h G_+}$.
Set
\begin{align*}
\Phi'_{\eps F + h G} & 
:=  \Psi_{\eps F -h G_-} + \Psi_{h G_+};
\\
\Phi_{F} & :=  \Phi_{(1-\eps)F} + \Psi_{\eps F -h G_-} 
+ \Psi_{hG_-};
\\
\Phi_{F+ h G} & :=  \Phi_{(1-\eps)F} + \Phi'_{\eps F + h G}, \end{align*}
so that $ \Phi'_{\eps F + h G} $, $\Phi_{F}$ and
$\Phi_{F+ h G}$
   are  Poisson processes with
distribution
$\Pi_{\eps F + h G}$,
$\Pi_{F }$,
and
$\Pi_{F + h G}$,
respectively.
Also, for $n \in \N$ define
$$
\Phi'_{h,n} := \Psi_{\eps F - h G_-}  
+  \Psi_{hG_+}|_{[B_n ]}
+ \Psi_{hG_-}|_{[\R^d \setminus B_n]},
$$
which is a Poisson process
with intensity $dx  \times (\eps F + h G ) (dr)$
in $[B_n ]$, and 
with intensity $dx  \times \eps F (dr)$
in $[\R^d \setminus B_n]$.
Since $F$ and $G$ are supported by $[0,1]$,
for $\psi \in \bN^1$ we have
\begin{equation}
Z(\Phi_{(1-\eps)F} + \Phi'_{h,n} + \psi )  \cap B_{n-1} =
Z(\Phi_{F + hG} + \psi ) \cap B_{n-1} .
\label{M9}
\end{equation}

Our next lemma gives us the first part (\ref{M2conj0}) of
Theorem \ref{3diff}, among other things.
\begin{lemma}
\label{Mlem}
  Let $L\subset\R^d$ be compact, and let $\psi \in \bN^1$ with
$\psi(\BX) < \infty$.
For $\varphi \in \bN$,
set $\tilde{f}_{L,\psi}(\varphi) := \I \{ L \cap Z_\infty(\varphi + \psi) \neq \emptyset\}$.
  Then 
\begin{align}
\left. 
  \frac{d^+}{d h}
\BE_{  F + hG } \tilde{f}_{L,\psi}(\Phi )
\right|_{h=0}  
=
 \iint
\BE_{F} D_{(x,r)} \tilde{f}_{L,\psi}(\Phi) 
G(dr)dx,
\label{M11}
 \end{align}
and the right-hand side of (\ref{M11}) is finite. Also,
  given $h \in [0,\eps^2]$ we have almost surely
\begin{equation}
\tilde{f}_{L,\psi} (\Phi_{ F + h G } )
 = \lim_{n \to \infty}
 \tilde{f}_{L,\psi}(\Phi_{(1- \eps) F}  + \Phi'_{h,n}).
\label{M1}
\end{equation}
\end{lemma}
\begin{proof}
To see
 (\ref{M1}),
 first suppose
$\tilde{f}_{L,\psi} (\Phi_{ F + h G } ) =1$.
If also
 $\Phi_{F+ h G} \in \bU $
(which is the case almost surely by (\ref{uic})),
there must be a path
 from $L$ through
$ Z(\Phi'_{\eps F + h G } + \psi)
\cup (Z(\Phi_{(1-\eps)F}) \setminus
 Z_\infty(\Phi_{(1-\eps)F})) $
   to $Z_\infty(\Phi_{(1-\eps)F})$
(if $L \cap Z_\infty(\Phi_{(1-\eps)F}) \neq \emptyset$
we interpret this path as being empty).
Choose such a path, and choose
$m \in \N$ such that this path is contained in $B_{m-1}$. Then
for $n \geq m$, by (\ref{M9}) we have 
 $\tilde{f}_{L,\psi}(\Phi_{(1- \eps) F}  + \Phi'_{h,n}) =1$.

Conversely,  suppose
$\tilde{f}_{L,\psi} (\Phi_{ F + h G } ) =0$.
Then, 
 recalling the definition (\ref{ZAdef}) of $Z_L(\varphi)$,
we have that
 $Z_L(\Phi_{F + h G } + \psi ) $ 
 is bounded
 so we can
choose $m$ such that
 $Z_L(\Phi_{F + h G } + \psi )  \subset 
B_{m-1}$.
Then for  $n \geq m$, by (\ref{M9}) we have
 $\tilde{f}_{L,\psi}(\Phi_{(1- \eps) F}  + \Phi'_{h,n}) =0$. Thus we have
demonstrated (\ref{M1}).

For $h \in [0,\eps]^2$ and $n \in \N$ set 
\begin{displaymath}
U_{h,n} = h^{-1} (
 \tilde{f}_{L,\psi} (\Phi_{(1-\eps)F} + \Phi'_{h,n} ) -   
 \tilde{f}_{L,\psi} (\Phi_{(1-\eps)F} + \Phi'_{h,n-1} ) ).   
\end{displaymath}
By (\ref{M1}) and dominated convergence,
$$
\BE \tilde{f}_{L,\psi}(\Phi_{F +h G}) = \lim_{n \to \infty} \BE \tilde{f}_{L,\psi}(\Phi_{(1-\eps)F} +
\Phi'_{h,n}).
$$
Also $\Phi_{(1-\eps)F} + \Phi'_{h,0} =\Phi_{F}$ almost surely.  Thus
\begin{equation}
h^{-1} (\BE \tilde{f}_{L,\psi}(\Phi_{F + hG}) -\BE  \tilde{f}_{L,\psi}(\Phi_{F}) ) 
= \sum_{n=1}^\infty
\BE \, U_{h,n} .
\label{M2}
\end{equation}

By Proposition \ref{russo}, for each $n$ we have
\begin{align}
\lim_{h \to 0+} \BE U_{h,n} 
 = &
\lim_{h \to 0+} h^{-1} \BE [
 \tilde{f}_{L,\psi} (\Phi_{(1-\eps)F}  + \Phi'_{h,n}  ) -  
 \tilde{f}_{L,\psi} (\Phi_{F}) ]
\nonumber
\\
& -   
\lim_{h \to 0+} h^{-1}  \BE [ 
 \tilde{f}_{L,\psi} (\Phi_{(1-\eps)F} + \Phi'_{h,n-1} )
- \tilde{f}_{L,\psi}(\Phi_{F}) 
 ]
\nonumber
\\
 = & ~\BE \int_{\R_+}  
\int_{B_{n} \setminus B_{n-1}}   D_{(x,r)} \tilde{f}_{L,\psi} ( \Phi_{F} )\, dx\, G(dr).  
\label{M3}
\end{align}
If we can take this limit through the sum on the right-hand side of
(\ref{M2}), then we have the desired result (\ref{M11}).
To justify this interchange, we seek to
dominate the terms  of (\ref{M2}) by those of a summable
sequence, independently of $h$.

Note that $|U_{h,n} | \leq h^{-1}$. Also
if $\Phi'_{h,n}  =  \Phi'_{h,n-1} $ then clearly $ U_{h,n} = 0$,
so that 
\begin{equation}
\{ U_{h,n} \neq 0 \} \subset 
\{ 
\Phi'_{h,n}  \neq  \Phi'_{h,n-1} \}.   
\label{M19}
\end{equation} 
Recall that we write $R_{K}$ for the radius of stabilization
$R_{K,1}$ defined at (\ref{rstabdef}).
 We assert the further event inclusion
\begin{equation}
\{ U_{h,n} \neq 0 \} \subset 
\{ 
R_L(\Phi_{(1-\eps)F},\Phi'_{\eps F + hG} +\psi )  > n-1  \} \cup 
\{\Phi_{F} \notin \bU \}.
\label{M10}
\end{equation}
To see this, let $n \in \N$, and
suppose $R_L(\varphi,\varphi') \leq n-1 $, where we now set
$$
\varphi= \Phi_{(1-\eps)F}, ~~~ 
\varphi' = \Phi'_{\eps F + h G} + \psi.
$$
By the definition of $R_L(\varphi,\varphi')$, 
we have either $R''_L(\varphi,\varphi') \leq n-1$, or
$R'_L(\varphi,\varphi') \leq n-1 $.
We show that in both cases $U_{h,n}=0$.

First suppose $R''_L(\varphi,\varphi') \leq n-1$.
In this case 
$L \cap Z_\infty(\varphi + \varphi'|_{[B_{n-1}]} ) \neq \emptyset$.
From the definitions at the start of this section, we have that
\bea
\Phi'_{h,n}|_{[B_{n}]}= \Phi'_{\eps F+ hG} |_{[B_{n}]};
~~~~~
\Phi'_{h,n-1}|_{[B_{n-1}]}= \Phi'_{\eps F+ hG} |_{[B_{n-1}]}.
\label{eq:new1}
\eea
Thus
$ \phi'|_{[B_{n-1}]} \leq  \Phi'_{h,n-1} + \psi$ and
 $ \phi'|_{[B_{n-1}]} \leq
  \phi'|_{[B_{n}]} \leq
 \Phi'_{h,n} +\psi$, so
 we obtain from the definition that
 $U_{h,n} = h^{-1} (1-1) =0$
 in this case.

Consider
the other case, with 
$R'_L(\varphi,\varphi') \leq n-1 $.
In this case
we have $Z_L(\varphi + \varphi') \subset B_{n-2}$.
Then any changes to $\varphi' $ outside
$[B_{n-1}]$ do not cause any change to
 $Z_L(\varphi + \varphi') $, so for all locally finite
 point measures $\chi$ we have
$$
Z_L(\varphi + \varphi'|_{[B_{n-1}]} + \chi|_{[\R^d 
\setminus B_{n-1} ]} )   \cap Z_\infty (\varphi) = \emptyset,
$$
and from this, using (\ref{eq:new1}) we can deduce that 
\begin{displaymath}
\tilde{f}_{L,\psi}(\phi + \Phi'_{h,n-1})
	 =  {\bf 1}\{ L \cap Z_\infty(\phi  + \phi'|_{[B_{n-1}]} 
	 + (\Phi'_{h,n-1}+\psi)|_{[\R^d \setminus B_{n-1}]}) 
	 \neq \emptyset \} =  0,
\end{displaymath}
and likewise $\tilde{f}_{L,\psi}(\phi + \Phi'_{h,n})=0$, so that
 $U_{h,n} = h^{-1} (0-0) =0$. This justifies (\ref{M10}).


 The event
 $\{\Phi'_{h,n} \neq\Phi'_{h,n-1}\} = \{(\Psi_{hG_+} + \Psi_{h G_-})
 ([B_n\setminus B_{n-1}])\neq 0\}$ (up to an event of probability
 zero) is independent of the event
 $\{R_L(\Phi_{(1-\eps)F},\Phi'_{\eps F + hG} + \psi) \leq n-1 \}$,
 by the stopping radius property (\ref{eqstop}),
and therefore also independent of the complementary event
 $\{R_L(\Phi_{(1-\eps)F},\Phi'_{\eps F + hG} + \psi) > n-1 \}$.
  Also
we have
the event inclusion:
\begin{align*}
\{ R_L(\Phi_{(1-\eps)F},\Phi'_{\eps F +h G} +\psi) > n-1
 \}
\subset \{ R_{L \cup Z(\psi)}
(\Phi_{(1-\eps)F},\Phi'_{\eps F +h G}) > 
n-1  \}.
\end{align*}
Therefore by (\ref{M19}),  (\ref{M10}) and  (\ref{uic}) we have
\begin{align*}
\BE | U_{ h,n} | 
 \leq 
 ~ 
  h^{-1} \BP \{ 
 \Phi'_{h,n}  \neq \Phi'_{h,n-1} \}\,
 \BP \{ R_{L \cup Z(\psi)}
(\Phi_{(1-\eps)F},\Phi'_{\eps F +h G}) > 
n-1  \}.
\end{align*}
Also there is a constant $c' \in (0,\infty)$ such that
 $\BP \{ 
 \Phi'_{h,n}  \neq \Phi'_{h,n-1} \} \leq n^{d-1} h c'$. 
Hence by Lemma \ref{lint}, there
is a  constant  $c \in (0,\infty)$ (independent of $n$ and $h$,
provided $0 \leq h \leq \eps^2/2$)
such that
\begin{align*}
\BE | U_{ h,n} | 
\leq  
 ~  c n^{d-1} \times \exp( - c^{-1} n )
\end{align*}
which is summable in $n$.  Hence by (\ref{M2}), (\ref{M3})
and dominated convergence we have (\ref{M11}).
This also shows that the right side of (\ref{M11}) 
is finite.
\end{proof}

\begin{proof}[Proof of Theorem~\ref{3diff}] 
We prove the result just for $b=1$.
The first part (\ref{M2conj0}) holds by
Lemma \ref{Mlem}. To prove the second part
 we take $F, G$ and $\eps$ as before but
now assume additionally that $F - \eps G$ is a measure.
Let $L \subset \R^d$ be compact.
As above,  
for each $\phi \in \bN$ and each  
   $\psi \in \bN^1$ with
$\psi(\BX) < \infty$,
set $\tilde{f}_{L,\psi}(\varphi) := 
\I \{ L \cap Z_\infty(\varphi + \psi) \neq \emptyset\}$, and set
 $\tilde{f}_L(\phi) := \tilde{f}_{L,0} (\varphi) =
 \I \{L \cap Z_\infty(\phi) \neq \emptyset\}$.
We shall prove by induction that for
 $n \in \N$ and 
$h \in (-\eps^2,\eps^2) $, we have
\begin{multline}
  \frac{d^{n}}{dh^n}\theta_L(  F + hG )
\\ =
\int \cdots \int
\BE_{F + hG } 
D^n_{(x_1,r_1),\ldots,(x_n,r_n)} \tilde{f}_L (\Phi) 
dx_1 G (dr_1) \cdots dx_n G(dr_n),
\label{M17}
\end{multline}
which implies (\ref{M2conj3}).

First consider $n=1$. Then (\ref{M17}) holds
for the right derivative 
at $h=0$ by Lemma \ref{Mlem}. Also, by applying
this fact to
$-G$ instead of $G$ we have that (\ref{M17}) holds
for the left derivative at $h=0$ too, so (\ref{M17}) holds
at $h = 0$. Therefore  (\ref{M17}) also holds
at other $h \in (-\eps^2,\eps^2)$ because we can
apply  the case $h=0$ of (\ref{M17}) to $F^*$ 
and $G^*$, given by  $F^*= F+h G$, and $G^* = G$.
Note that $F^*$
 is strictly supercritical because $F^* = (1-\eps) F
+ \eps(F+ (h/\eps) G)$ and
$ \eps(F+ (h/\eps) G)$ is a measure since $|h| < \eps^2$.

Now we perform the inductive step.
Let
 $n \in \N$, and
 suppose (\ref{M17}) holds for 
 all $h \in (-\eps^2,\eps^2)$.
Then for $0 < h < \eps^2 $,
\begin{align}
h^{-1}
   \Bigl(\left. \frac{d^{n}}{ds^n}\theta_L(  F + sG )\right|_{s=h}
-
  \left. \frac{d^{n}}{ds^n}\theta_L(  F + sG )\right|_{s=0}
\Bigr)
\nonumber \\
= \int \cdots \int u_{x_1,r_1,\ldots,x_n,r_n} (h)\,  dx_1 G(dr_1) \cdots dx_n G(dr_n),
\label{M14}
\end{align}
where we set
\begin{eqnarray*}
u_{x_1,r_1,\ldots,x_n,r_n} (h) := h^{-1} \left( 
\BE D^n_{(x_1,r_1),\ldots,(x_n,r_n)}\tilde{f}_L(\Phi_{F + hG}) 
\right.
\\
-
\left. \BE D^n_{(x_1,r_1),\ldots,(x_n,r_n)}\tilde{f}_L(\Phi_{F }) 
\right).
\end{eqnarray*}
Applying Lemma \ref{Mlem} to the function
$D^n_{(x_1,r_1),\ldots,(x_n,r_n)}\tilde{f}_L$ 
(expressed as a sum as in (\ref{differn2})) gives us
as $h \to 0$ that
\begin{align}
u_{x_1,r_1,\ldots,x_n,r_n} (h) \to \int \int 
\BE D^{n+1}_{(x,r),(x_1,r_1),\ldots,(x_n,r_n)} \tilde{f}_L(\Phi_{F}) 
\,dx\, G(dr). 
\label{M16}
\end{align}
For $1 \leq i \leq n$,
write $z_i$ for $(x_i,r_i)$.
By (\ref{M1}) applied to $\tilde{f}_{L,\psi}$ for each $\psi \in \bN^1$
with $\psi \leq \sum_{i=1}^n \delta_{z_i}$, and dominated convergence,
we have for $(z_1,\ldots,z_n) \in \BX^n$
 and $|h| < \eps^2$ that
\begin{align}
u_{z_1,\ldots,z_n}(h)
 & = h^{-1} \left(  \lim_{m \to \infty}
\BE D^n_{z_1,\ldots,z_n} \tilde{f}_{L}(\Phi_{(1-\eps)F} + \Phi'_{h,m}) 
- \BE D^n_{z_1,\ldots,z_n} \tilde{f}_{L}(\Phi_{F} ) \right) 
\nonumber \\
& = \sum_{m=1}^\infty  \BE V(h, m,z_1,\ldots,z_n),
\label{M15}
\end{align}
where we set
\begin{multline}
  V(h, m,z_1,\ldots,z_n)
  := h^{-1} \Bigl(
    D^n_{z_1,\ldots,z_n}\tilde{f}_{L}(\Phi_{(1-\eps) F} + \Phi'_{h,m} ) \\-
    D^n_{z_1,\ldots,z_n}\tilde{f}_{L}(\Phi_{(1-\eps) F} + \Phi'_{h,m-1} )\Bigr). \label{eq:2}
\end{multline}
Now, $|V(h,m,z_1,\ldots,z_n)| \leq 2^{n+1} h^{-1}$
and clearly we have
\begin{align}
\{ V(h,m,z_1,\ldots,z_n) \neq 0 \} \subset \{ \Phi'_{h,m} \neq 
\Phi'_{h,m-1} \}.  
\label{M19a}
\end{align}


Set  $M = \max (m, |x_1|,\ldots,|x_n| )$. 
Suppose $M \geq (2n+4)  (\diam (L \cup \{0\}) +4 )$.
Choose $I \in \{1,\ldots, 2n+3\}$ such that the annulus
$$
\Gamma_{M,n,i} := B_{(I+1)M/(2n+4)} \setminus B_{IM/(2n+4)} 
$$
intersects none of the balls  $B_{r_1}(x_1),\ldots,B_{r_n}(x_n)$
and also
does not intersect the annulus $B_{m+1} \setminus  B_{m-2}$;
to be definite, choose the smallest such $I$. 

Cover the boundary $\partial B_{(I+0.5)M/(2n+4)} $ of the ball
$B_{(I+0.5)M/(2n+4)} $ with a deterministic collection of unit balls
$C_1, \ldots C_{k(M)}$, each $C_j$ with center $w_j$ in
$\partial B_{(I+0.5)M/(2n+4)} $, with $k(M) = O(M^{d-1})$. Let
$\tau_j(\varphi)$ denote the shifted point measure 
$\varphi (\cdot + w_j)$. Define the event
\begin{align}
A'_{h,m,z_1,\ldots,z_n} :=
\cup_{j=1}^{k(M)}
\{ R_{B_1}
(\tau_j(\Phi_{(1-\eps)F}), \tau_j(\Phi'_{h,m}))
\geq (M/(4n+8)-2) \}.
	\label{e:A'}
\end{align}
We remark that our definition  of
$A'_{h,m,z_1,\ldots,z_n} $ is different from the one suggested in
\cite{LPZ24}; this is because we subsequently found some further 
issues with verifying  \eqref{M12} below using that definition.

Write just $\Phi$ for $\Phi_{(1-\eps)F}$.
Note that the  events $ \{ \Phi'_{h,m} \neq \Phi'_{h,m-1} \} $ and  
$A'_{h,m,z_1,\ldots,z_n}$ are independent. 
Indeed, the first event is determined by $\Psi_{hG_+}|_{[B_m \setminus 
B_{m-1}]}$ and
$\Psi_{hG_-}|_{[B_m \setminus B_{m-1}]}$, while the second 
event is determined by $\Phi$,
$\Psi_{\eps F- hG}$,
and the restrictions of 
$\Psi_{hG_-}$ and $\Psi_{hG_+}$ to $[\Gamma_{M,n,I}]$,
and the annuli $B_{m} \setminus B_{m-1}$
and $\Gamma_{M,n,I}$
are disjoint.
We  will show next  that
\begin{align}
\{ V(h,m,z_1,\ldots,z_n) \neq 0 \} 
\subset 
A'_{h,m,z_1,\ldots,z_n}\,.
\label{M12}
\end{align}
To justify this, observe first that
by the definition of $M$, at least one of the sets
$B_{r_1}(x_1),\ldots,$ $ B_{r_n}(x_n),$ $ B_{m+2}\setminus B_{m-1}$ is
 {\em exterior} to $B_{(I+1)M/(2n+4)}$, i.e.\ has empty intersection
 with it.

Suppose that $A'_{h,m,z_1,\ldots,z_n}$ does not occur.
Suppose also that one of the balls $B_{r_i}(x_i)$
(say the ball $B_{r_1}(x_1)$) is exterior to
 $B_{(I+1)M/(2n+4)}$; then for any 
$\psi \in \bN$ with $\psi \leq \sum_{i=2}^n \delta_{z_i}$
we claim that
\begin{align}
\tilde{f}_{L, \psi}(\Phi + \Phi'_{h,m})
= \tilde{f}_{L, \psi}(\Phi + \Phi'_{h,m} + \delta_{z_1}).
\label{e:Dx1}
\end{align}
Indeed, suppose otherwise;
then the left-hand side
of \eqref{e:Dx1} equals zero, while the right-hand side equals 1.
Then there must exist a bounded component
of $Z(\Phi + \Phi'_{h,m}+ \psi)$ that connects $L$ to
$B_{r_1}(x_1)$. This component must cross the annulus
$\Gamma_{M,n,I}$ and therefore passes through 
$C_j$ for some $j \leq k(M)$; hence for this choice of $j$ we have
$R'_{B_1}(\tau_j(\Phi),\tau_j(\Phi'_{h,m})) \geq (M/(4n+8)) -2$.
Therefore since we assume 
 $A'_{h,m,z_1,\ldots,z_n}$ does not occur, we must have
$R''_{B_1}(\tau_j(\Phi),\tau_j(\Phi'_{h,m})) < (M/(4n+8)) -2$.
Since $B(w_j,(M/(4n+8))-2) \subset \Gamma_{M,n,i}$,
this implies
\begin{align}
C_j \cap Z_\infty(\Phi + \Phi_{h,m}|_{
	[\Gamma_{M,n,I}]}) 
\neq \emptyset,
	\label{e:shiftR''}
\end{align}
and therefore $L$ is connected (via $C_j$) to 
$Z_\infty(\Phi + \Phi_{h,m} + \psi)$
so the left side of \eqref{e:Dx1} equals 1, which is a contradiction.

By \eqref{e:Dx1} we have
 $D^n_{z_1,\ldots,z_n}\tilde{f_L}
(\Phi + \Phi'_{h,m})  =0$, and we can show similarly that
\linebreak
 $D^n_{z_1,\ldots,z_n}\tilde{f_L}
(\Phi + \Phi'_{h,m-1})  =0$, so that
$V(h,m,z_1,\ldots,z_n)=0$ in this case.

Now suppose instead that the annulus $B_{m+1} \setminus B_{m-2}$
is exterior to  $B_{(I+1)M/(2n+4)}$, and as before that
$A'_{h,m,z_1,\ldots,z_n}$ does not occur. 
Let
$\psi \in \bN$ with $\psi \leq \sum_{i=1}^n \delta_{z_i}$.
Then we claim that
\begin{align}
\tilde{f}_{L, \psi}(\Phi + \Phi'_{h,m})
= \tilde{f}_{L, \psi}(\Phi + \Phi'_{h,m-1} ) .
\label{e:DM}
\end{align}
Indeed, suppose 
$\tilde{f}_{L, \psi}(\Phi + \Phi'_{h,m}) =1$ and
$ \tilde{f}_{L, \psi}(\Phi + \Phi'_{h,m-1} )  =0$. 
Then there must be a bounded component of $Z(\Phi + 
(\Phi'_{h,m} + \psi)|_{[B_{m-1}]})$ connecting $L$
to $\R^{d} \setminus B_{m-2}$. This component
must pass through $C_j$ for some $j \leq k(M)$, so that
for this $j$
(as in the previous case)
$R'_{B_1}(\tau_j(\Phi),\tau_j(\Phi'_{h,m})) \geq (M/(4n+8)) -2$
and hence 
$R''_{B_1}(\tau_j(\Phi),\tau_j(\Phi'_{h,m})) < (M/(4n+8)) -2$,
so that \eqref{e:shiftR''} holds.
But this implies
$L$ is connected (via $C_j$) to 
$Z_\infty(\Phi + (\Phi_{h,m} + \psi)|_{[B_{m-1}]})
= Z_\infty(\Phi + (\Phi_{h,m-1} + \psi)|_{[B_{m-1}]})$,
so both sides of  \eqref{e:DM} are equal to 1, which is a contradiction.
We can obtain a contradiction similarly if
$\tilde{f}_{L, \psi}(\Phi + \Phi'_{h,m}) =0$ and
$ \tilde{f}_{L, \psi}(\Phi + \Phi'_{h,m-1} )  =1$,
so we have justified the claim \eqref{e:DM}.
By \eqref{e:DM} we have 
$V(h,m,z_1,\ldots,z_n)=0$ in this case too.
Together with the previous two paragraphs, this implies 
 the assertion (\ref{M12}).

We show next that  for $n$ and $L$ fixed 
there is a constant $c$ such that for all $m$,
all  $z_1,\ldots,z_n
\in \R^d \times [0,1]$ and all $h \in (-\eps^2/2,\eps^2/2)$  we have
\begin{align}
\label{M13}
 \BP[ A'_{h,m,z_1, \ldots, z_n} 
 ] \leq 
c \exp ( -c^{-1} M).
\end{align}
Assume again that $M \geq (2n +4 ) (\diam (L \cup \{0\}) + 4 )$.
For each $j$ let $C'_j: = B(w_j, M/(4n+8)) $.
Since 
   the restriction of Poisson process  
$\Phi'_{h,m}$
or $\Phi'_{h,m-1}$
to $[C'_j]$ has intensity
of product form
(either $dx  \times (\eps F + h G ) (dr)$
or $dx  \times \eps F (dr)$, depending
on whether or not the annulus $B_{m}  
\setminus B_{m-1} $ is exterior to $B_{(I+1)M/(2n+4)}$), 
 we can use the union bound and Lemma~\ref{lint} to obtain~(\ref{M13}). 

Using (\ref{M19a}),  (\ref{M12})
  and (\ref{M13}), we obtain 
that there is a finite constant (again denoted $c$, and  depending on $n$) such that
\begin{align*}
	\BE | V(h,m,z_1,\ldots,z_n) |
	&\leq \BP \{ \Phi'_{h,m} \neq \Phi'_{h,m-1} \}
	\BP \{ A'_{h,m,z_1,\ldots,z_n} \}
	\\
	& \leq c m^{d-1} \exp( -( m+
\max (|x_1|,\ldots,|x_n|) ) /c ),
\end{align*}
which is summable in $m$ with the sum being integrable in
$(z_1,\ldots,z_n)$.  Then using (\ref{M15}) and dominated convergence
we can take the limit (\ref{M16}) inside the integral (\ref{M14}), so
that
\begin{align*}
\left.
 \frac{d^+}{dh}
 \frac{d^{n}}{dh^n}\theta_L(  F + hG )
\right|_{h=0} 
=
\BE_{F  } 
\int \cdots \int
D^{n+1}_{(x,r),(x_1,r_1),\ldots,(x_n,r_n)} \tilde{f}_{L} (\Phi) 
\\
\times
dx G(dr)\,
dx_1 G (dr_1) \cdots dx_n G(dr_n).
\end{align*}
Also we can repeat this argument using $-G$ 
instead of $G$ to get the same value  for the left derivative at $h=0$
leading to~(\ref{M17}) for $n+1$ with  $h=0$.
Then for $n+1$ and for a general $h \in (-\eps^2,\eps^2)$
we have (\ref{M17}) by applying the $h=0$ result and using
the  measure $F+h G$ instead of $F$. This completes the induction.
\end{proof}

\section{Proof of Theorem \ref{texponent}}
\label{texproof}

Given a graph $\bG =(V,E)$, and given $v \in V$, let us denote by
$\bG \setminus v$ the graph $\bG$ with $v$ and all edges incident to
$v$ removed.  If $u,v,w$ are distinct vertices of $\bG$, let us say
vertex $w $ is $(u,v)$-pivotal if $u$ and $v$ lie in the same
component of $\bG$ but different components of $\bG \setminus w$.
\begin{lemma}
  \label{pivlemma}
  Suppose $\bG = (V,E)$ is a finite connected graph, and $u,v \in E$ 
with $u \neq v$.
  Then either $\bG$ has at least one $(u,v)$-pivotal vertex, or there
  exist at least two vertex-disjoint paths in $\bG$ from $u$ to $v$.
  Also, in the first case, every path from $u$ to $v$ in $\bG$ passes
  through the $(u,v)$-pivotal vertices in the same order.
\end{lemma}
\begin{proof}
  The first assertion is an immediate consequence of Menger's theorem
  (see e.g. \cite{Boll79}).

  To see the second assertion, suppose $w,w'$ are distinct
  $(u,v)$-pivotal vertices, and there is a path from $u$ to $v$
  passing through $w$ before $w'$, and another such path passing
  through $w'$ before $w$. Then following the first path from $u$ as
  far as $w$, and then the second path from $w$ to $v$, we obtain a
  path from $u$ to $v$ avoiding $w'$; hence $w'$ is not
  $(u,v)$-pivotal, which is a contradiction.
\end{proof}

\begin{proof}[Proof of Theorem~\ref{texponent}] 
  Our proof uses ideas from \cite{ChayesCh86}.  We start by introducing
  some notation.
Fix $b >0$, $F \in \bM_1^b$ and compact $L \subset \R^d$.
  Since $F$ is fixed, we write $\theta_L(t)$ for
  $\theta_L(tF)$ in this proof. By a {\em path} in a configuration
  $\varphi\in \bN$ we mean a finite or infinite sequence
  $K_1,K_2,\dots$ of distinct grains such that $K_i=B_{r_i}(x_i)$ for
  some $(x_i,r_i)\in\varphi$ and $K_i\cap K_{i+1}\ne\emptyset$ for all
  $i \ge 1$ with $K_{i+1}$ part of the sequence.  A path {\em
    intersects} a subset of $\R^d$, if one of its constituent grains
  intersects this set.  If $A,A'$ are disjoint subsets of $\R^d$ and a
  path intersects both $A$ and $A'$, then we say the path {\em joins}
  $A$ to $A'$.  We shall say that two paths $(K_1,K_2,\ldots)$ and
  $(K'_1,K'_2,\ldots)$ in $\varphi$ are {\em disjoint} if
  $K_i \neq K'_j$ for all $i,j$.

  For
 $n \in \N$ introduce events that
  there is a path joining $L$ to the complement of $B_n$ or to
  infinity,
  \begin{displaymath}
    J_L^n :=\{\phi\in\bN:\   Z_L(\phi) \setminus B_n \neq\emptyset\};
    \quad
    J_L:=\{\phi\in\bN:\ L\cap  Z_\infty(\phi)\neq\emptyset\}. 
  \end{displaymath}
where the notation $Z_L$ is as defined in (\ref{ZAdef}).
Assume from now on that $n$ is so large that
 $L \subset B_{n-2b}$.  If $\phi\in J_L^n$, but
  $(x,r)\in\phi$ is such that $\phi-\delta_{(x,r)}\not\in J_L^n$, we
  say that the grain $B(x,r)$ is {\em pivotal} for $J_L^n$ in the
  configuration $\phi$.

  Let $\theta_L^n(t):=\BP_{tF}\{\Phi\in J_L^n\}$.  We claim that when
  $\varphi \in J_L^n$, either there are at least two disjoint paths
  from $L$ to $\R^d \setminus B_n$ in $\varphi$ or there is at least
  one pivotal grain for $J_L^n$ and there is a unique {\em last
    pivotal grain} for $J_L^n$ when counting from $L$.  To see this
  claim, apply Lemma~\ref{pivlemma} to the intersection graph of the
  set
  \begin{equation}\label{eq:defbn}
    \mathbf{B}_n=\{ B_{r_i} (x_i): (x_i ,r_i) \in \varphi\ \text{such
      that}\ B_{r_i}(x_i)\cap B_n\ne\emptyset\}  \cup
    \{L, \R^d \setminus B_n\}.
  \end{equation}

Let $\varphi \in  \bN$ and  $n  \in \N$. Suppose there is
a unique last pivotal grain for 
   $J_L^n$  and denote this last pivotal grain by
$K := B_r(x)$. If
  $K \subset B_n$, then there exist three disjoint paths in
  $\varphi - \delta_{(x,r)}$: one which joins $L$ to $K$ and two which
  join $K$ to $\R^d \setminus B_n$. If not (i.e.\ if $K \setminus B_n 
\neq \emptyset$), there is still a path
  joining $K$ to $L$.  Even in this case we say
  that there are two disjoint paths joining $K$ to
  $\R^d \setminus \partial B_n$ which are just both empty; see
  Figure~\ref{fig:pivotals}.
    \begin{figure}[ht]
      \centering 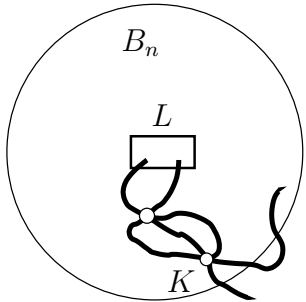
    \caption{Geometry of the paths (depicted as black curves) connecting
     $L$ to $\R^d\setminus B_n$. Pivotal grains for $J_L^n$ are coloured white.
     The last pivotal grain starting from $L$ is denoted $K$.
     \label{fig:pivotals}}
 \end{figure}
  We then have
  \begin{multline*}
    \theta_L^n(t)=\BP_{tF}\{\text{there are two disjoint paths in 
$\Phi$  joining $L$ to $\R^d \setminus B_n$}\}\\
    + \BE_{tF}\int\I\{\text{$B_r(x)$ is the last pivotal grain for
      $J_L^n$ in $\Phi$}\}\,\Phi(d(x,r))\\
    \leq (\theta_L^n(t))^2 + t
\iint 
\BP_{tF}\{\text{$B_r(x)$ is the
      last pivotal grain for $J_L^n$ in $\Phi+\delta_{(x,r)}$}\}
\\
\times
\, dx\,F(dr),
  \end{multline*}
  where we have used the B-K inequality \cite[Th.~2.3]{MeesterRoy96}
  to bound the first term from above and the Mecke identity
  (\ref{eq:mecke}) for the second term. Now $B_r(x)$ is the last
  pivotal grain for $J_L^n$ in $\Phi+\delta_{(x,r)}$ if and only if
  there are two disjoint paths in the
 configuration $\Phi+\delta_{(x,r)}$,
  one of them joining $B_r(x)$ and $\R^d \setminus B_n$ (possibly
  empty, if $B_r(x) \setminus B_n \neq \emptyset$), and the other one
  joining $L$ and $\R^d \setminus B_n$ and using $B_r(x)$, and all
  paths joining $L$ to $\R^d \setminus B_n$ use $B_r(x)$.  We claim
  that this is the same as saying the events
  \begin{displaymath}
    E_{n,x,r} := \{\phi:\ \varphi + \delta_{(x,r)} \in J_L^n, \phi \notin J_L^n \}
  \end{displaymath}
  and $J_{B_r(x)}^n$ occur disjointly
 in the sense of \cite{GuRao99},
 so that by the continuum Reimer
  inequality in that paper, we get
  \begin{equation}\label{eq:33}
    \theta_L^n (t)\leq (\theta_L^n (t))^2
    + t \iint \BP_{tF} \{ \Phi \in E_{n,x,r}\}\,
    \BP_{tF} \{\Phi\in J_{B_r(x)}^n\}\, dx\,F(dr).
  \end{equation}

  Let us justify our claim.  With probability 1, there exists a finite
  (random) $\varepsilon$ such that displacement of the locations $x_i$
  with $(x_i,r_i) \in \Phi|_{[B_{n+2b} ]} $ by at most $\varepsilon$,
  and modification of the corresponding $r_i$ by at most
  $\varepsilon$, would not affect the intersection graph on
 $\mathbf{B}_n \cup B_r(x)$.
  Suppose $\Phi$ is such that there are disjoint paths $P,P'$ in
  configuration $\Phi+\delta_{(x,r)}$, with $P$ joining $B_r(x)$ and
  $\R^d \setminus B_n$ and $P'$ joining $L$ and $\R^d \setminus B_n$
  and using $B_r(x)$, and suppose also all paths joining $L$ to
  $\R^d \setminus B_n$ use $B_r(x)$. Let $\eps$ be as defined
  above. Let ${\cal K}$ be a union of rational $(d+1)$-cubes of side
  less than $\varepsilon /(d+1)$ centered on the points $(x_i,r_i)$ such
  that $B_{r_i}(x_i) \in P$.  Let ${\cal L}$ be the complement of
  ${\cal K}$ in $\R_d \times \R_+$.

  If we modify our configuration $\Phi$ arbitrarily in ${\cal L}$ then
  we are still in $J^n_{B_r(x)}$, since the points of $\Phi$ inside
  ${\cal K}$ guarantee occurrence of $J^n_{B_r(x)}$.

  On the other hand, if we modify $\Phi$ arbitrarily in ${\cal K}$
  then we still have a path joining $L$ to $\R^d \setminus B_n$ using
  $B_r(x)$ (because our configuration in ${\cal L}$ contains such a
  path) but every such path uses $B_r(x)$ (because in $\Phi$ our path
  $P$ did not intersect with any path joining $L$ to $B_r(x)$, and
  hence, by the choice of $\varepsilon$, neither does any modification
  of $P$ by moving points a distance at most $\varepsilon/(d+1)$ in
  each coordinate, and the rest of $\Phi$ is unchanged).

  Note that our regions ${\cal K},{\cal L}$ are unions of rational
  rectangles in $(d+1)$-space, not in $d$-space as in
  \cite{GuRao99}.  To see that \cite{GuRao99} is applicable, 
 note that we can generate our Poisson process
  $\Phi = \sum_i \delta_{(x_i,r_i)}$ in $ \R^d \times \R_+$ from a
  homogeneous point process $\sum_{i} \delta_{y_i}$ of intensity $t$
  in $\R^d \times [0,1]$, with the spatial locations $x_i$ generated
  by projecting $y_i$ onto the first $d$ coordinates and the random
  radii $r_i$ generated as a suitable increasing function (namely, the
  quantile function of $F$) of the final coordinate of $y_i$.

  Now let $n\to\infty$. Since $(J_L^n)_{n \geq 1}$ is a decreasing
  sequence of events and $\cap_n J_L^n=J_L$, we have
  $\theta_L^n (t)\to \theta_L(t)$, and for every $\phi\in\bN$ we have
 \begin{displaymath}
   \I\{\phi \in E_{n,x,r}\}
   \to \I\{\phi+\delta_{(x,r)}\in J_L \}
    \I\{\phi \notin J_L\}.
  \end{displaymath}
  Also
  $\BP_{tF}\{\Phi\in J_{B_r(x)}^n\}\to \BP_{tF}\{\Phi\in J_{B_r(x)}\}
  =\theta_{B_r}(t)$ by stationarity.

By the definition (\ref{rstabdef}),
the first factor of the integrand in~\eqref{eq:33} satisfies
\begin{displaymath}
    \BP_{tF} \{ \Phi \in E_{n,x,r} \}
    \le  \BP_{tF} \{ |x|-r \leq R_{L,b}(0,\Phi) < \infty
    \}.
  \end{displaymath}
  Note that $R''_{L,b}(0,\Phi) = \infty$ and therefore
  $R_{L,b}(0,\Phi) = R'_{L,b}(0,\Phi) = R'_{L,b}(\Phi,0)$.  The above
  inequality is true because if $Z_L(\Phi)\cap B_r(x) = \emptyset $,
  then $B_r(x)$ cannot be pivotal for $J_L^n$ in
  $\Phi+\delta_{(x,r)}$.  Indeed, there must be a path joining $L$ to
  $B_r(x)$ to give $B_r(x)$ a chance of being pivotal, but if this
  path is part of $Z_\infty(\Phi)$ then $B_r(x)$ is not pivotal.

Recall that we are assuming
  $F((b,\infty))=0$.  By (\ref{rstabdef}) we have
  $Z_L(\Phi) \subset B_{R_{L,b}(\Phi,0)} $, $\BP_{tF}$-almost
  surely.  Hence by (\ref{MPL2}) from Lemma \ref{lint} the integrand
  in (\ref{eq:33}) is bounded by an integrable function of $(x,r)$.
  Hence, by (\ref{eq:33}) and dominated convergence, setting
  $\theta'_L(t):=\frac{d}{dt}\theta_L(t) $ we have
  \begin{align}
    \theta_L (t)\le & (\theta_L (t))^2 + t \iint \BP_{tF} 
                      \{\Phi + \delta_{(x,r)} \in J_L, \Phi \notin J_L
                      \}\,
                      \theta_{B_r}(t) \, dx\,F(dr)
                      \notag
    \\
    \le & (\theta_L(t))^2 + t\theta_{B_b}(t) \theta'_L (t),
          \label{2}
  \end{align}
  where for the last line we have used Theorem \ref{diff} and the
  monotonicity of $\theta_{B_r}(t)$ in $r$.

  Next we bound $\theta_L(t) / \theta_{B_b}(t)$ from below.  Pick
  $x \in L$.  If $B_b(x) \cap Z_\infty (\Phi) \neq \emptyset$ and also
  $B_b(x) \subset Z(\Phi)$, then
  $L \cap Z_\infty(\Phi) \neq \emptyset$.  Hence by the Harris-FKG
  inequality for Poisson processes (see, e.g.\ \cite[Ch.20]{LaPe16}),
  and translation-invariance, we have for $t\ge t_c$
  that
  \begin{equation}
    \theta_L(t) \ge \theta_{B_b(x)}(t)\BP_{tF}\{B_b(x) \subset Z\} =
    \theta_{B_b}(t)\BP_{t_c F}\{B_b \subset Z\}.
    \label{eq:1}
  \end{equation}
  Setting $\alpha =\BP_{t_cF}\{B_b\subset Z\}$ and
  substituting~\eqref{eq:1} into~\eqref{2},
  we get to
 \begin{displaymath}
   \theta'_L(t) \geq \frac{(1-\theta_L(t))\theta_L(t) }{t \theta_{B_b}(t) }
   \geq \frac{\alpha (1-\theta_L(t))}{t}.
 \end{displaymath}

 Integrating over $t$, using the continuity of $\theta'_L(\cdot)$
on $(t_c,\infty)$
and the monotonicity of $\theta_L(\cdot)$,
 we therefore have 
 that
  \begin{displaymath}
    \theta_L(t)-\theta_L(t_c)\geq \alpha (t-t_c)\frac{1-\theta_L(t)}{t}
 ,
  \end{displaymath}
which is \eqref{eq:rate1}. Since $\theta_L(t)$ is continuous from the right,
  $\theta_L(t)=\theta_L(t_c)+o(1)$ as $t\downarrow t_c$, giving~\eqref{eq:rate2}.
\end{proof}

\section{Final remarks}
\label{secfinal}

In this paper we have studied the capacity functional of the 
infinite cluster of a spherical Boolean model. Our main results
(Theorems \ref{diff} and \ref{texponent}) require the radii
to be deterministically bounded. It can be expected that these
results also hold for more general Boolean models with
connected grains having a deterministically bounded circumradius.
It can also be conjectured that good moment properties
of the circumradius should suffice to imply the result for 
unbounded radii. The proof
of this latter extension, however, does not seem to be 
straightforward.

The methods of this paper can probably be used to
derive differentiability properties of the expectations of 
other functionals of the Boolean model. 
A whole family of such functionals in the subcritical regime can be defined
in terms of the number $N_r$, $r>0$,  of grains
in the cluster of $Z(\Phi+\delta_{(0,r)})$,  intersecting  the
ball $B_r$. 
Given 
$F \in \bM_1^\sharp$ and $m \in \Z$,
it is then of interest to study the
functional $\int \BE_{tF} [N_r^m]\,F(dr)$ as a function of $t<t_c$. 
In the case $m=-1$ this is the mean number of clusters
per a typical Poisson point. Preliminary results in the latter
case can be found in \cite{JiZhG11}.

A natural step after the infinite differentiability would be to show
that the capacity is an analytic function of intensity in the
supercritical phase.  It might be possible to use
(\ref{M18}) to show that for fixed supercritical $t$,
the Taylor series for $\theta_L((t+h)F)$ as a function
of $h$ has positive radius of convergence; however this seems
to need tighter bounds han those used here, and hence,
new ideas.  

Also of interest is the {\em $n$-point connectivity function}
of the Boolean model. Given $x_1,\ldots,x_n \in \R^d$, and given 
$F \in \bM^\sharp_1$, for $t >0$
let $\tau_{x_1,\ldots,x_n}(t)$ denote
the $\BP_{tF}$-probability that  the points $x_1, \ldots, x_n$
all lie in the same component
of $Z(\Phi)$. It is not hard to prove
 that $\tau_{x_1,\ldots,x_n}(t)$ is continuous
in $t$  (see, e.g.,~\cite{JiZhG11}). Using the method of proof of
Theorem~\ref{3diff}, it should be possible to show further that
$\tau_{x_1,\ldots,x_n} (t)$ is infinitely differentiable in
$t$ on the interval $(t_c(F),\infty)$.  
Moreover, the $n$-point connectivity function of
$Z_\infty(\Phi)$
(as opposed to that of $Z(\Phi)$) is
certainly infinitely differentiable,
since by the inclusion-exclusion formula the
probability $\BP_{tF}(\cap_{i=1}^n\{x_i \in Z_\infty(\Phi)\})$
can be expressed as a linear combination of the
capacity functionals of the subsets of $\{x_1,\ldots,x_n\}$,
plus a constant.

It may be possible to generalize Theorem~\ref{texponent} as follows.
Let $F_0 \in \bM^\sharp$ and
 $F \in \bM_1^\sharp.$
Let $t_c(F_0,F)$ be
the supremum of those  $t$ such that $\theta(F_0+tF) =0$
and assume $F_0,F$ are such that $t_c(F_0,F) >0$. 
Then we might expect that  similar results to (\ref{eq:rate1}) and
(\ref{eq:rate2})  would hold with $tF$ replaced
by $F_0+tF$ (and $t_c F$ replaced by $F_0+ t_cF$ 
and $t_1 F$ replaced by $F_0 + t_1F$
 wherever they  appear.

We have shown that $\theta_L$ is (under the
assumptions of Theorem \ref{diff}) infinitely differentiable
on $(t_c,\infty)$. 
It would be
extremely interesting to understand the behaviour
of the second derivative near the critical value.
We leave this as a challenging problem for future research.

\bigskip
\noindent{\bf Acknowledgment.}
  The mistake in the original paper \cite{LPZ17} was originally
  brought to our attention by Emil Schmitz, to whom we are indebted
  for pointing this out to us.

  GL and MP were supported by the German Research Foundation (DFG)
  through the research unit "Geometry and Physics of Spatial Random
  Systems" under the grant HU 1874/3-1. MP was supported by EPSRC grant EP/T028653/1.


\end{document}

%% file: pivotals1.pdf_t
\begin{picture}(0,0)%
\includegraphics{pivotals1.pdf}%
\end{picture}%
\setlength{\unitlength}{4144sp}%
\begingroup\makeatletter\ifx\SetFigFont\undefined%
\gdef\SetFigFont#1#2#3#4#5{%
  \reset@font\fontsize{#1}{#2pt}%
  \fontfamily{#3}\fontseries{#4}\fontshape{#5}%
  \selectfont}%
\fi\endgroup%
\begin{picture}(1786,1808)(252,-1176)
\put(1126,-106){\makebox(0,0)[lb]{\smash{{\SetFigFont{12}{14.4}{\rmdefault}{\mddefault}{\updefault}{\color[rgb]{0,0,0}$L$}%
}}}}
\put(946,344){\makebox(0,0)[lb]{\smash{{\SetFigFont{12}{14.4}{\rmdefault}{\mddefault}{\updefault}{\color[rgb]{0,0,0}$B_n$}%
}}}}
\put(1216,-1096){\makebox(0,0)[lb]{\smash{{\SetFigFont{12}{14.4}{\rmdefault}{\mddefault}{\updefault}{\color[rgb]{0,0,0}$K$}%
}}}}
\end{picture}%